\documentclass[1p,fleqn,11pt]{elsarticle}

\usepackage{amsmath,amsthm,amsfonts,amscd}
\usepackage{eucal} 	 	
\usepackage{verbatim}      	
\usepackage{makeidx}       	
\usepackage{./setup/psfig}         	
\usepackage{epsfig}         	
\usepackage{./setup/citesort}         	%
\usepackage{url}		

\newtheorem{thm}{Theorem}[section]

\newtheorem{prop}[thm]{Proposition}

\theoremstyle{definition}

\theoremstyle{remark}

\usepackage{amsmath}
\usepackage{amssymb}
\usepackage{amsfonts}
\usepackage{amsthm}
\usepackage{amscd}
\usepackage[normalem]{ulem}
\usepackage{graphicx}
\usepackage[hang,small]{caption}
\usepackage[labelformat=simple]{subcaption}

\usepackage{wrapfig}
\usepackage{multirow}
\usepackage{xparse}
\usepackage{environ}
\usepackage{stackengine,wasysym}
\usepackage{todonotes}
\makeatletter
\renewcommand{\todo}[2][]{\tikzexternaldisable\@todo[#1]{#2}\tikzexternalenable}
\makeatother

\usepackage{tensor}
\usepackage{varwidth,calc}

\usepackage{pgfplotstable}
\usepackage{pgfplots}
\pgfplotsset{
	colormap={my basis colormap}{
		rgb255=(0, 114, 189);
		rgb255=(54, 106, 148);
		rgb255=(108, 98, 107);
		rgb255=(163, 91, 66);
		rgb255=(217, 83, 25);
	}
}
\pgfplotsset{
	colormap={my parula}{
		rgb255=(53.0655, 42.406, 134.9460);
		rgb255=(20.2336, 132.6061, 211.9511);
		rgb255=(55.5463, 184.8859, 157.9118);
		rgb255=(208.7187, 186.8470,  89.1966);
		rgb255=(248.9565, 250.6905, 13.7190);
	}
}
\pdfpageattr {/Group << /S /Transparency /I true /CS /DeviceRGB>>}
\usepackage{tikz}
\usepgfmodule{oo}
\usetikzlibrary{calc}
\usetikzlibrary{shapes}
\usetikzlibrary{plotmarks}
\usetikzlibrary{backgrounds}
\usetikzlibrary{decorations.pathmorphing,decorations.pathreplacing,decorations.markings}
\usetikzlibrary{arrows.meta,arrows,fit,matrix,positioning,shapes.geometric,positioning}
\usetikzlibrary{external}
\usetikzlibrary{cd}
\AtBeginEnvironment{tikzcd}{\tikzexternaldisable}
\AtEndEnvironment{tikzcd}{\tikzexternalenable}

\usepackage{appendix}
\usepackage{chngcntr}
\usepackage{etoolbox}
\usepackage{apptools}
\AtAppendix{\counterwithin{theorem}{section}}
\AtBeginEnvironment{subappendices}{%
	\chapter*{Appendix}
	\addcontentsline{toc}{chapter}{Appendices}
}

\usepackage{breakcites}
\usepackage[inline]{enumitem}

\usepackage{scalerel,stackengine,wasysym}

\usepackage{accents}

\usepackage{algorithm}
\usepackage[noend]{algpseudocode}

\DeclareMathAlphabet{\mathpzc}{OT1}{pzc}{m}{it}

\usepackage{color}
\definecolor{lightgray}{gray}{0.80}
\definecolor{lightgray}{gray}{0.80}

\usepackage{mdframed}

\usepackage[titles]{tocloft}
\cftsetindents{figure}{0em}{3.5em}
\cftsetindents{table}{0em}{3.5em}
\renewcommand\geq\geqslant
\renewcommand\leq\leqslant

\usepackage{titlesec}
\titleformat{\chapter}[block]
{\normalfont\huge\bfseries}{\thechapter.}{1em}{\huge}
\titlespacing*{\chapter}{0pt}{-19pt}{0pt}

\usepackage[T3,T1]{fontenc}

\newcommand{\Bezier}{B\'ezier~}

\theoremstyle{plain}
\newtheorem{theorem}{Theorem}[section]
\newtheorem{proposition}[theorem]{Proposition}
\newtheorem{lemma}[theorem]{Lemma}
\newtheorem{corollary}[theorem]{Corollary}

\theoremstyle{remark}
\newtheorem{remark}[theorem]{Remark}

\newtheorem{examplex}[theorem]{Example}
\newenvironment{example}
{\pushQED{\qed}\examplex}
{\popQED\endexamplex}

\theoremstyle{definition}
\newtheorem{definition}[theorem]{Definition}
\newtheorem{assumption}[theorem]{Assumption}

\usepackage{mathtools}

\newcommand{\mbf}[1]{{\boldsymbol{#1}}}
\newcommand{\RR}{{\mathbb{R}}}

\newcommand{\PP}{{\mathcal{P}}}
\newcommand{\ZZ}{{\mathbb{Z}}}
\newcommand{\ZZP}{\mathbb{Z}_{\geq0}}

\newcommand{\dimwp}[1]{\dim{\left(#1\right)}}

\newcommand{\kerwp}[1]{\ker{\left(#1\right)}}
\newcommand{\moplus}{\mathop{\oplus}}

\newcommand{\isomorphic}{\cong}

\stackMath
\newcommand\reallywidehat[1]{%
	\savestack{\tmpbox}{\stretchto{%
			\scaleto{%
				\scalerel*[\widthof{\ensuremath{#1}}]{\kern-.6pt\bigwedge\kern-.6pt}%
				{\rule[-\textheight/2]{1ex}{\textheight}}
			}{\textheight}%
		}{0.5ex}}%
	\stackon[1pt]{#1}{\tmpbox}%
}
\stackMath
\newcommand\reallywidecheck[1]{%
	\savestack{\tmpbox}{\stretchto{%
			\scaleto{%
				\scalerel*[\widthof{\ensuremath{#1}}]{\kern-.6pt\bigvee\kern-.6pt}%
				{\rule[-\textheight/2]{1ex}{\textheight}}
			}{\textheight}%
		}{0.5ex}}%
	\stackon[1pt]{#1}{\tmpbox}%
}

\ExplSyntaxOn
\NewEnviron{bmatrixT}
{
	\marine_transpose:V \BODY
}

\int_new:N \l_marine_transpose_row_int
\int_new:N \l_marine_transpose_col_int
\seq_new:N \l_marine_transpose_rows_seq
\seq_new:N \l_marine_transpose_arow_seq
\prop_new:N \l_marine_transpose_matrix_prop
\tl_new:N \l_marine_transpose_last_tl
\tl_new:N \l_marine_transpose_body_tl

\cs_new_protected:Nn \marine_transpose:n
{
	\seq_set_split:Nnn \l_marine_transpose_rows_seq { \\ } { #1 }
	\int_zero:N \l_marine_transpose_row_int
	\prop_clear:N \l_marine_transpose_matrix_prop
	\seq_map_inline:Nn \l_marine_transpose_rows_seq
	{
		\int_incr:N \l_marine_transpose_row_int
		\int_zero:N \l_marine_transpose_col_int
		\seq_set_split:Nnn \l_marine_transpose_arow_seq { & } { ##1 }
		\seq_map_inline:Nn \l_marine_transpose_arow_seq
		{
			\int_incr:N \l_marine_transpose_col_int
			\prop_put:Nxn \l_marine_transpose_matrix_prop
			{
				\int_to_arabic:n { \l_marine_transpose_row_int }
				,
				\int_to_arabic:n { \l_marine_transpose_col_int }
			}
			{ ####1 }
		}
	}
	\tl_clear:N \l_marine_transpose_body_tl
	\int_step_inline:nnnn { 1 } { 1 } { \l_marine_transpose_col_int }
	{
		\int_step_inline:nnnn { 1 } { 1 } { \l_marine_transpose_row_int }
		{
			\tl_put_right:Nx \l_marine_transpose_body_tl
			{
				\prop_item:Nn \l_marine_transpose_matrix_prop { ####1,##1 }
				\int_compare:nF { ####1 = \l_marine_transpose_row_int } { & }
			}
		}
		\tl_put_right:Nn \l_marine_transpose_body_tl { \\ }
	}
	\begin{bmatrix*}[r]
		\l_marine_transpose_body_tl
	\end{bmatrix*}
}
\cs_generate_variant:Nn \marine_transpose:n { V }
\cs_generate_variant:Nn \prop_put:Nnn { Nx }
\ExplSyntaxOff

\makeatletter
\catcode`! 3
\def\Transpose #1{\romannumeral0\expandafter
	\Mar@Transpose@a\romannumeral`^^@\Mar@DoOneRow #1\\!\\}

\def\Mar@DoOneRow #1\\{\Mar@DoOneRow@a {}#1&^^@&}%

\def\Mar@DoOneRow@a #1#2&{%
	\if^^@\detokenize{#2}\expandafter\@gobble\fi
	\Mar@DoOneRow@a {#1#2\\}%
}%

\def\Mar@Transpose@a #1#2\\{\ifx!#2\expandafter\Mar@FinishTranspose\fi
	\expandafter\Mar@Transpose@b\romannumeral`^^@\Mar@DoOneRow@a {}#2&^^@&#1}

\def\Mar@Transpose@b #1#2^^@\\{\Mar@Join {}#2^^@!#1}

\def\Mar@Join #1#2\\#3!#4\\%
{\if^^@\detokenize{#3}\expandafter\Mar@EndJoin\fi
	\Mar@Join {#1#2&#4\\}#3!}%

\def\Mar@EndJoin\Mar@Join #1^^@!^^@\\{\Mar@Transpose@a {#1^^@\\}}

\def\Mar@FinishTranspose
#1&^^@&#2\\^^@\\{ #2}

\catcode`! 12
\makeatother

\newcommand{\face}{\sigma}
\newcommand{\edge}{\tau}
\newcommand{\vertex}{\gamma}

\newcommand{\interior}[1]{\accentset{\circ}{#1}}
\newcommand{\mesh}{\mathcal{T}}
\newcommand{\meshF}{\mesh_2}
\newcommand{\meshE}{\mesh_1}
\newcommand{\meshV}{\mesh_0}
\newcommand{\meshInterior}{\interior{\mesh}}
\newcommand{\meshInteriorE}{\meshInterior_1}
\newcommand{\meshInteriorV}{\meshInterior_0}

\newcommand{\boundary}{\partial}
\newcommand{\domain}{\Omega}
\newcommand{\domainInterior}{\interior{\domain}}

\newcommand{\meshEH}{\tensor*[^h]{\mesh}{_1}}
\newcommand{\meshEV}{\tensor*[^v]{\mesh}{_1}}
\newcommand{\meshInteriorEH}{\tensor*[^h]{\meshInterior}{_1}}
\newcommand{\meshInteriorEV}{\tensor*[^v]{\meshInterior}{_1}}

\newcommand{\ncells}[1]{\mathfrak{t}_{#1}}

\newcommand{\degreeu}{m}
\newcommand{\degreev}{m'}
\newcommand{\smooth}{r}
\newcommand{\bsmooth}{\mbf{r}}

\newcommand{\splSpace}{\mathcal{R}}

\newcommand{\bsmoothr}{{\mbf{s}}}

\newcommand{\elem}[1]{[#1]}
\newcommand{\faceE}[1]{\elem{\face_{#1}}}
\newcommand{\edgeE}[1]{\elem{\edge_{#1}}}
\newcommand{\vertexE}[1]{\elem{\vertex_{#1}}}

\newcommand{\ideal}[1][I]{\mathfrak{#1}}
\newcommand{\idealComplex}{\mathcal{I}}
\newcommand{\constantComplex}{\mathcal{C}}
\newcommand{\quotientComplex}{\mathcal{Q}}

\newcommand{\euler}[1]{\chi \left( #1 \right)}

\definecolor{myBlue}{rgb} {0,0.4470,0.7410}
\definecolor{myRed}{rgb} {0.8500,0.3250,0.0980}
\definecolor{myGray1}{rgb} {0,0,0}
\definecolor{myGray2}{rgb} {0.6,0.6,0.6}
\definecolor{myGray3}{rgb} {0.45,0.45,0.45}
\definecolor{myGray4}{rgb} {0.8,0.8,0.8}
\definecolor{myGray5}{rgb} {0.55,0.55,0.55}

\definecolor{myCPlot1}{rgb} {0,    0.4470,    0.7410}
\definecolor{myCPlot2}{rgb} {0.8500,   0.3250,    0.0980}
\definecolor{myCPlot3}{rgb} {0.9290,    0.6940,    0.1250}
\definecolor{myCPlot4}{rgb} {0.4940,    0.1840,    0.5560}
\definecolor{myCPlot5}{rgb} {0.4660,    0.6740,    0.1880}
\definecolor{myCPlot6}{rgb} {0.3010,    0.7450,    0.9330}
\definecolor{myCPlot7}{rgb} {0.6350,    0.0780,    0.1840}

\definecolor{myCPColor}{rgb} {0.9255, 0.6941, 0.355}

\tikzstyle{scpColor}=[circle, fill=myCPColor]

\tikzset{
	show curve controls/.style={
		decoration={
			show path construction,
			curveto code={
				\draw[myGray1,dashed,eThickness]
				(\tikzinputsegmentfirst)
				-- (\tikzinputsegmentsupporta)
				-- (\tikzinputsegmentsupportb)
				-- (\tikzinputsegmentlast)
				;
				\fill[scpColor] (\tikzinputsegmentfirst) circle(3pt);
				\fill[scpColor] (\tikzinputsegmentsupporta) circle(3pt);
				\fill[scpColor] (\tikzinputsegmentsupportb) circle(3pt);
				\fill[scpColor] (\tikzinputsegmentlast) circle(3pt);
				\draw[#1,line width=1pt]
				(\tikzinputsegmentfirst)
				.. controls (\tikzinputsegmentsupporta)
				and (\tikzinputsegmentsupportb) ..
				(\tikzinputsegmentlast);
			}
		},decorate
	}
}

\tikzset{
	on each segment/.style={
		decorate,
		decoration={
			show path construction,
			moveto code={},
			lineto code={
				\path [#1]
				(\tikzinputsegmentfirst) -- (\tikzinputsegmentlast);
			},
			curveto code={
				\path [#1] (\tikzinputsegmentfirst)
				.. controls
				(\tikzinputsegmentsupporta) and (\tikzinputsegmentsupportb)
				..
				(\tikzinputsegmentlast);
			},
			closepath code={
				\path [#1]
				(\tikzinputsegmentfirst) -- (\tikzinputsegmentlast);
			},
		},
	},
	mid arrow/.style={postaction={decorate,decoration={
				markings,
				mark=at position .6 with {\arrow[#1]{Latex}}
	}}},
}

\tikzset{
	bThickness/.style={line width=#1\pgflinewidth},
	bThickness/.default={2},
}

\tikzset{
	eThickness/.style={line width=#1\pgflinewidth},
	eThickness/.default={0.5},
}

\pgfooclass{myAxis}{ 
	\method myAxis() { 
	}
	\method plot(#1,#2,#3,#4,#5,#6,#7) { 
		\draw[-latex,thick,myBlue] (#1,#2) --+ (#4:#3) node [right,sloped,black] {$#6$};
		\draw[-latex,thick,myRed] (#1,#2) --+ ({#4+#5}:#3) node [above,sloped,black] {$#7$};
	}
}
\pgfoonew \myAxis=new myAxis()

\tikzset{cross/.style={cross out, draw, 
		minimum size=2*(#1-\pgflinewidth), 
		inner sep=0pt, outer sep=0pt}}
	
\pgfooclass{myNewAxis}{ 
	\method myNewAxis() { 
	}
	\method plotX(#1,#2,#3,#4) { 
		\draw[-latex,thick,myBlue] (#1) --+ (#3:#2) node [right,sloped,black] {$#4$};
	}
	\method plotY(#1,#2,#3,#4,#5) {
		\draw[-latex,thick,myRed] (#1) --+ ({#3+#4}:#2) node [above,sloped,black] {$#5$};
	}
}
\pgfoonew \myNewAxis=new myNewAxis()

\pgfooclass{myBezierElements}{ 
	\method myBezierElements() { 
	}
	\method plot(#1,#2,#3,#4,#5,#6,#7,#8,#9) { 
		\begin{scope}[rotate=#6,xslant=cot(180-#5),xscale=1,yscale=sin(#5)]
			\draw[eThickness, step={#4/#1}, #7, #8, #9] (#2,#3) grid ({#2+#4},{#3+#4}); 
		\end{scope}
	}
	\method plotPoints(#1,#2,#3,#4,#5,#6,#7,#8,#9) { 
		\begin{scope}[rotate=#6,xslant=cot(180-#5),xscale=1,yscale=sin(#5)]
			\draw[step=#4, #7, black] (#2,#3) grid (#2+#4,#3+#4); 
			\def\dx{#4*1.0/3}
			\draw[step=\dx, #8, #9, black] (#2,#3) grid (#2+#4,#3+#4); 
			\foreach \x in {1,...,#1}{
				\foreach \y in {1,...,#1}{
					\node [draw, black, fill=white, shape=rectangle, minimum width=3pt, minimum height=3pt, anchor=center,scale=0.4] at (#2+\x*\dx-\dx,#3+\y*\dx-\dx) {};
				}
			}
		\end{scope}
	}
	\method plotGrid(#1,#2,#3,#4,#5,#6,#7,#8,#9) { 
		\begin{scope}[rotate=#6,xslant=cot(180-#5),xscale=1,yscale=sin(#5)]
			\draw[step=#4, #7, black] (#2,#3) grid (#2+#4,#3+#4); 
			\def\dx{#4/3}
			\draw[step=\dx, #8, #9, black] (#2,#3) grid (#2+#4,#3+#4); 
		\end{scope}
	}
	\method plotPointsWLabels(#1,#2,#3,#4,#5,#6) { 
		\begin{scope}[rotate=#6,xslant=cot(180-#5),xscale=1,yscale=sin(#5)]
			\draw[step=#4, thick, black] (#2,#3) grid (#2+#4,#3+#4); 
			\def\dx{#4/3}
			\draw[step=\dx, ultra thin, black] (#2,#3) grid (#2+#4,#3+#4); 
			\foreach \x in {1,...,#1}{
				\foreach \y in {1,...,#1}{
					\node [draw, myBlue, shape=rectangle, minimum width=3pt, minimum height=3pt, anchor=center] at (#2+\x*\dx-\dx,#3+\y*\dx-\dx) {
						\pgfmathparse{\x-1}
						\pgfmathprintnumber[    
						fixed,
						fixed zerofill,
						precision=0
						]{\pgfmathresult}%
						\pgfmathparse{\y-1}
						\pgfmathprintnumber[    
						fixed,
						fixed zerofill,
						precision=0
						]{\pgfmathresult}%
					};
				}
			}
		\end{scope}
	}
	\method plotVertices(#1,#2,#3,#4,#5,#6,#7,#8) { 
		\begin{scope}[rotate=#6,xslant=cot(180-#5),xscale=1,yscale=sin(#5)]
			\pgfmathsetmacro{\ncp}{#1+1}
			\foreach \x in {1,...,\ncp}{
				\foreach \y in {1,...,\ncp}{
					\node[draw,black,circle,fill=#7,scale=#8] at ({#2+(\x-1)*#4/#1},{#3+(\y-1)*#4/#1}) {};
				}
			}	
		\end{scope}
	}
	\method shade(#1,#2,#3,#4,#5,#6,#7) { 
		\begin{scope}[rotate=#6,xslant=cot(180-#5),xscale=1,yscale=sin(#5)]
			\fill[#7] (#2,#3) rectangle +(#4,#4);
		\end{scope}
	}
	\method plotC1Vertices(#1,#2,#3,#4,#5,#6,#7,#8) {
		\begin{scope}[rotate=#6,xslant=cot(180-#5),xscale=1,yscale=sin(#5)]
			\pgfmathsetmacro{\nel}{#1}
			\foreach \i in {1,...,\nel}{
				\foreach \j in {1,...,\nel}{
					\pgfmathsetmacro{\xl}{#2+(\i-1)*#4/#1}
					\pgfmathsetmacro{\xr}{#2+(\i)*#4/#1}
					\pgfmathsetmacro{\yt}{#3+(\j)*#4/#1}
					\pgfmathsetmacro{\yb}{#3+(\j-1)*#4/#1}
					\foreach \ii in {1,...,2}{
						\foreach \jj in {1,...,2}{
							\pgfmathsetmacro{\x}{(\xl*(3-\ii)+\xr*(\ii))/3}
							\pgfmathsetmacro{\y}{(\yb*(3-\jj)+\yt*(\jj))/3}
							\node[draw,black,regular polygon,regular polygon sides=4,fill=#7,scale=#8] at (\x,\y) {};
						}
					}
				}
			}
		\end{scope}
	}
}
\pgfoonew \myBezierElements=new myBezierElements()

\pgfooclass{myQuadElement}{ 
	\method myQuadElement() { 
	}
	\method plot(#1,#2,#3,#4,#5,#6,#7,#8) { 
		\begin{scope}
			\draw[eThickness,black] (#1,#2) -- (#3,#4) -- (#5,#6) -- (#7,#8) -- cycle;
			\node[coordinate] (v1) at (#1,#2) {};
			\node[coordinate] (v2) at (#3,#4) {};
			\node[coordinate] (v3) at (#5,#6) {};
			\node[coordinate] (v4) at (#7,#8) {};
		\end{scope}
	}
	\method plotColor(#1,#2,#3,#4,#5,#6,#7,#8,#9) { 
		\begin{scope}
			\draw[eThickness,dashed,#9] (#1,#2) -- (#3,#4) -- (#5,#6) -- (#7,#8) -- cycle;
			\node[coordinate] (v1) at (#1,#2) {};
			\node[coordinate] (v2) at (#3,#4) {};
			\node[coordinate] (v3) at (#5,#6) {};
			\node[coordinate] (v4) at (#7,#8) {};
		\end{scope}
	}
	\method plotInteriorDomainPoints(#1,#2,#3,#4,#5,#6) {
		\begin{scope}
			\pgfmathtruncatemacro{\p}{#5-1}
			\foreach \x in {2,...,\p}{
				\foreach \y in {2,...,\p}{
					\pgfmathsetmacro{\a}{(\x-1)/(#5-1)}
					\pgfmathsetmacro{\b}{(\y-1)/(#5-1)}
					\pgfmathsetmacro{\aone}{(1-\a)*(1-\b)}
					\pgfmathsetmacro{\atwo}{(\a)*(1-\b)}
					\pgfmathsetmacro{\athree}{(\a)*(\b)}
					\pgfmathsetmacro{\afour}{(1-\a)*(\b)}
					\node [draw, black, shape=rectangle, minimum width=3pt, minimum height=3pt, anchor=center,scale=0.5,fill=#6] (b\x\y) at (barycentric cs:#1=\aone,#2=\atwo,#3=\athree,#4=\afour) {};
				}
			}
		\end{scope}
	}
	\method plotEdgeDomainPoints(#1,#2,#3,#4,#5) {
		\begin{scope}
			\pgfmathtruncatemacro{\p}{#3-1}
			\foreach \x in {2,...,\p}{
				\pgfmathsetmacro{\a}{(\x-1)/(#3-1)}
				\pgfmathsetmacro{\aone}{(1-\a)}
				\pgfmathsetmacro{\atwo}{\a}
				\node [draw, black, shape=rectangle, minimum width=3pt, minimum height=3pt, anchor=center,scale=0.5,fill=#5] (e#4\x) at (barycentric cs:#1=\aone,#2=\atwo) {};
			}
		\end{scope}
	}
	\method plotCorner(#1,#2,#3) {
		\begin{scope}
			\node [draw, black, shape=rectangle, minimum width=3pt, minimum height=3pt, anchor=center,scale=0.5,fill=#3] (c#2) at (#1) {};
		\end{scope}
	}
}
\pgfoonew \myQuad=new myQuadElement()

\pgfplotstableset{
	create on use/hthree/.style={
		create col/copy column from table={/Users/deepesh/Documents/research/reports/utdiss2-05/from_dropbox/tikz/extraordinary_points/data/meshSize3.txt}{H}
	}
}
\pgfplotstableset{
	create on use/hfive/.style={
		create col/copy column from table={/Users/deepesh/Documents/research/reports/utdiss2-05/from_dropbox/tikz/extraordinary_points/data/meshSize5.txt}{H}
	}
}
\pgfplotstableset{
	create on use/hsix/.style={
		create col/copy column from table={/Users/deepesh/Documents/research/reports/utdiss2-05/from_dropbox/tikz/extraordinary_points/data/meshSize6.txt}{H}
	}
}
\pgfplotstableset{
	create on use/hseven/.style={
		create col/copy column from table={/Users/deepesh/Documents/research/reports/utdiss2-05/from_dropbox/tikz/extraordinary_points/data/meshSize7.txt}{H}
	}
}
\pgfplotstableset{
	create on use/hmep/.style={
		create col/copy column from table={/Users/deepesh/Documents/research/reports/utdiss2-05/from_dropbox/tikz/extraordinary_points/data/meshSizemep.txt}{H}
	}
}

\makeatletter
\def\ps@pprintTitle{%
	\let\@oddhead\@empty
	\let\@evenhead\@empty
	\def\@oddfoot{}%
	\let\@evenfoot\@oddfoot}
\makeatother

\usepackage[hidelinks]{hyperref}

\begin{document}
	
	\title{Dimension of polynomial splines of mixed smoothness on T-meshes}
	\corref{cor1}
	\author[delft]{Deepesh Toshniwal}
	\ead{d.toshniwal@tudelft.nl}	
	\cortext[cor1]{Corresponding author}
	\author[swansea]{Nelly Villamizar}
	\ead{n.y.villamizar@swansea.ac.uk}
	\address[delft]{Delft Institute of Applied Mathematics, Delft University of Technology}
	\address[swansea]{Department of Mathematics, Swansea University}

	\begin{abstract}
		In this paper we study the dimension of splines of mixed smoothness on axis-aligned T-meshes.
		This is the setting when different orders of smoothness are required across the edges of the mesh.
		Given a spline space whose dimension is independent of its T-mesh's geometric embedding, we present constructive and sufficient conditions that ensure that the smoothness across a subset of the mesh edges can be reduced while maintaining stability of the dimension.
		The conditions have a simple geometric interpretation.
		Examples are presented to show the applicability of the results on both hierarchical and non-hierarchical T-meshes.
		For hierarchal T-meshes it is shown that mixed smoothness spline spaces that contain the space of PHT-splines (Deng et al., 2008) always have stable dimension.
	\end{abstract}
	
	\begin{keyword}
		splines \sep T-meshes \sep mixed smoothness \sep dimension formula \sep homological algebra
	\end{keyword}
	\maketitle
	
	\section{Introduction}\label{sec:introduction}
Polynomial splines on polyhedral partitions are ubiquitous in approximation theory, geometric modelling, and computational analysis. 
It is customary to ask splines to be $C^r$ smooth across all mesh facets for a fixed choice of $r \in \ZZ_{\geq -1}$ that depends on the intended application.
However, certain applications also require working with splines for which smoothness can be reduced across an arbitrary subset of the mesh facets; e.g., to model non-smooth or even discontinuous geometric features.
Such splines will be said to have mixed smoothness, and they constitute the focus of this article.

\noindent\emph{Example:} 
(Application to fluid flows around thin solids).
Consider the case of a thin solid immersed in an incompressible fluid flow, and a numerical simulation that employs a solid-conforming mesh, i.e., a mesh where the solid is modeled as the union of a subset of the facets.
In general, we would like to use smooth splines for approximating the fluid pressure and velocity fields.
However, unless the discrete pressure field is allowed to be discontinuous across the thin solid, the simulation results would be meaningless.
At the same time, we would like to retain smoothness of the pressure field across the remaining facets.
See \cite{sauer2018monolithic} for an example of such an application.\hfill{$\blacklozenge$}

An appealing feature of splines in applications is the flexibility in
the choice of the underlying meshes.
In particular, there is a rich history of the use of simplicial, quadrilateral and cuboidal meshes for uniform polynomial degrees and a fixed order of global smoothness, see e.g., \cite{cirak2000subdivision,hughes2005isogeometric}.
Univariate spline spaces and the construction of a suitable spline basis for them, called the B-spline basis, are well understood, see \cite{deboor1978practical} for example.
A spline basis for tensor product spline spaces can be easily defined on tensor product quadrilateral meshes by taking tensor products of univariate B-splines; this process can be directly extended to higher dimensions for building multivariate spline spaces.
A comprehensive overview of splines on triangulations can be found in \cite{lai2007spline} and the references therein.

When applications require the resolution of the spline space to be increased on a subset of the mesh faces, the most common approach is to employ local subdivision.
Spline constructions on such locally subdivided meshes have been proposed in \cite{schumaker2012splines} for triangulations and in \cite{sederberg2003t,giannelli2012,dokken2013} for quadrilateral meshes, among others.
We will focus on the case of locally subdivided quadrilateral meshes, the so-called T-meshes. Examples of such meshes will be discussed in Section \ref{sec:examples}.

The study of multivariate splines, and bivariate splines on T-meshes in particular, poses an interesting challenge as the spline space dimension can depend on the geometric embedding of the mesh \cite{li_instability_2011}.
In practice, identifying meshes where the dimension is stable -- i.e., free from this dependence -- is useful for avoiding cases where spline spaces on combinatorially and topologically equivalent meshes have different dimensions.
Several techniques have been used for studying the dimension of multivariate splines.
We will do so for splines of mixed smoothness using the homology-based approach introduced in \cite{billera1988homology}, and therefore in the following we stick to a brief discussion of the same.

By interpreting splines as the top homology of a chain complex, Billera \cite{billera1988homology} used tools from homological algebra for studying the dimension of splines.
Modifications of the complexes proposed by Schenck and Stillman \cite{schenck1997family,schenck1997local} have since been used by Mourrain and Villamizar \cite{mourrain2013homological} for bounding the spline space dimension on simplicial meshes in two and three dimensions.
Schenck and Sorokina \cite{schenck_subdivision_2018} have recently studied the problem on simplicial meshes where one maximal face has been subdivided.
On T-meshes, Mourrain \cite{mourrain2014dimension} provided bounds on the dimension of bi-degree $(\degreeu,\degreev)$ splines.
Generalizations of the bounds from \cite{mourrain2013homological,mourrain2014dimension} to splines with local polynomial degree adaptivity been recently provided in \cite{toshniwal_polynomial_2018,toshniwal_polynomial_2019}.

Let $\splSpace^\bsmooth_{\degreeu\degreev}$ denote the space of bi-degree $(\degreeu, \degreev)$ splines that are $\bsmooth(\edge)$ smooth across mesh edge $\edge$.
As stated above, we will use homology-based techniques similar to the ones used in \cite{billera1988homology,schenck1997local,mourrain2014dimension} to study the dimension of $\splSpace^\bsmooth_{\degreeu\degreev}$.
Then, given that $\splSpace^\bsmooth_{\degreeu\degreev}$ has stable dimension, we provide sufficient conditions for preservation of this stability when the desired orders of smoothness are decreased across a subset of the mesh edges.
Let us denote this latter spline space with $\splSpace^\bsmoothr_{\degreeu\degreev}$, with $\bsmoothr(\edge) \leq \bsmooth(\edge)$ for all edges $\edge$.
Note that in general the results proposed in \cite{mourrain2014dimension} cannot be applied to compute the dimension of $\splSpace^\bsmoothr_{\degreeu\degreev}$.
This is because they require the smoothness across all horizontal (resp. vertical) edges that form a connected union to be the same; we do not impose the same restriction here.
Instead of studying $\splSpace^\bsmoothr_{\degreeu\degreev}$ from scratch, we use information from $\splSpace^\bsmooth_{\degreeu\degreev}$ to considerably simplify the problem.
In particular, in Section \ref{sec:dimension} we provide sufficient conditions that ensure that the dimension of $\splSpace^\bsmoothr_{\degreeu\degreev}$ can be computed combinatorially using local information only.
The conditions are constructive in nature and have a simple geometric interpretation.
Application of the results to both hierarchical and non-hierarchical T-meshes are presented in Section \ref{sec:examples}.

	\section{Preliminaries: splines, meshes and homology}\label{sec:preliminaries}

This section will introduce the relevant notation that we will use for working with polynomial splines on T-meshes.

\subsection{Splines on T-meshes}\label{ss:meshes_and_splines}
\begin{definition}[T-mesh]\label{def:tmesh}
	A T-mesh $\mesh$ of $\RR^2$ is defined as:
	\begin{itemize}
		\item a finite collection $\meshF$ of axis-aligned rectangles $\face$ that we consider as open sets of $\RR^2$ having non-zero measure, called $2$-cells or faces, together with
		\item a finite set $\meshE$ of closed axis-aligned segments $\edge$, called $1$-cells, which are edges of the (closure of the) faces $\sigma\in\meshF$, and
		\item the set $\meshV$, of vertices $\vertex$, called $0$-cells,  of the edges $\tau\in\meshE$, 
	\end{itemize}
such that the following properties are satisfied:
	\begin{itemize}
		\item $\face \in \meshF \Rightarrow$ the boundary $\boundary\face$ of $\sigma$ is a finite union of edges in $\meshE$,
		\item $\face, \face' \in \meshF \Rightarrow \face \cap \face' = \boundary\face \cap \boundary\face'$ is a finite union of edges in $\meshE \cup \meshV$, and,
		\item $\edge, \edge' \in \meshE \text{ with } \edge \neq \edge' \Rightarrow \edge\cap \edge' 
		 \in \meshV$.
	\end{itemize}
	The domain of the T-mesh is assumed to be connected and is defined as $\domain := \cup_{\face\in\meshF}\face \subset \RR^2$.
\end{definition}

Sets of horizontal and vertical edges will be denoted by $\meshEH$ and $\meshEV$, respectively.
Edges of the T-mesh are called interior edges if they intersect the interior of the domain of the T-mesh $\domainInterior$. Otherwise, they are called  boundary edges. 
The set of interior edges will be denoted by $\meshInteriorE$; and the sets of interior horizontal and vertical edges will be denote by $\meshInteriorEH$ and $\meshInteriorEV$, respectively.
Similarly, if a vertex is in $\domainInterior$ it will be called an interior vertex, and a boundary vertex otherwise.
The set of interior vertices will be denoted by $\meshInteriorV$.
We will denote the number of $i$-cells with $\ncells{i} := \# \mesh_i$.


\begin{assumption}\label{ass:simplyConnectedDomain}
	The domain $\domain$ is simply connected, and $\domainInterior$ is connected. 
\end{assumption}

A T-mesh which satisfies Assumption \ref{ass:simplyConnectedDomain} will be said to be simply connected. 
We define $\PP_{\degreeu\degreev} $ as the vector space of polynomials of bi-degree at most $(\degreeu,\degreev)$ spanned by the monomials $s^i t^j$, $0 \leq i \leq \degreeu$ and $0 \leq j \leq \degreev$.
If either of $\degreeu$ or $\degreev$ are negative, then $\PP_{\degreeu\degreev} := 0$.
The final ingredient that we need for defining a spline space on $\mesh$ is a smoothness distribution on its edges.

\begin{definition}[Smoothness distribution]
	The map $\bsmooth : \meshE \rightarrow \ZZ_{\geq -1}$ is called a smoothness distribution if $\bsmooth(\edge) = -1$ for all $\edge \notin \meshInteriorE$.
\end{definition}

Using this notation, we can define the spline space $\splSpace^\bsmooth_{\degreeu\degreev}$ that forms the object of our study.
From the following definition and the definition of $\bsmooth$, it will be clear that we are interested in obtaining highly local control over the smoothness of splines in $\splSpace^\bsmooth_{\degreeu\degreev}$, a feature that is missing from the existing literature which studies spline on T-meshes.

\begin{definition}[Spline space]
	Given mesh $\mesh$, bi-degree $(\degreeu,\degreev) \in \ZZP^2$, smoothness distribution $\bsmooth$, we define the spline space $\splSpace^\bsmooth \equiv \splSpace^{\bsmooth}_{\degreeu\degreev}(\mesh)$ as
	\begin{equation}
	\begin{split}
		\splSpace^{\bsmooth}_{\degreeu\degreev}(\mesh) := \bigg\{f\colon &\forall \face \in \meshF~~f|_\face \in \PP_{\degreeu\degreev}\;,\\
		&\forall \edge \in \meshInteriorE~~f \in C^{\bsmooth(\edge)}\text{~smooth across~}\edge\bigg\}\;.
	\end{split}
	\end{equation}
\end{definition}

From the above definition, the pieces of all splines in $\splSpace^\bsmooth$ are constrained to meet with smoothness $\bsmooth(\edge)$ at an interior edge $\edge$; we will also define \begin{equation*}
	\bsmooth_h(\vertex) := \min_{\substack{\edge \ni \vertex\\\edge \in \meshEV}} \bsmooth(\edge)\;,\qquad
	\bsmooth_v(\vertex) := \min_{\substack{\edge \ni \vertex\\\edge \in \meshEH}} \bsmooth(\edge)\;.
\end{equation*}
We will use the following algebraic characterization of smoothness in this document.
\begin{lemma}[Billera \cite{billera1988homology}]\label{lem:smoothness}
	For $\face, \face' \in \meshF$, let $\face \cap \face' = \edge \in \meshInteriorE$. A piecewise polynomial function equalling $p$ and $p'$ on $\face$ and $\face'$, respectively, is at least $\smooth$ times continuously differentiable across $\edge$ if and only if 
	\begin{equation*}
		\ell_\edge^{\smooth+1}~\big|~p - p'\;,
	\end{equation*}
	where $\ell_\edge$ is a non-zero linear polynomial vanishing on $\edge$.
\end{lemma}

In line with the above characterization and for each interior edge $\edge$, we define $\ideal^\bsmooth_\edge$ to be the vector subspace of $\PP_{\degreeu\degreev}$ that contains all polynomial multiples of $ \ell_\edge^{\bsmooth(\edge)+1}$; when $\bsmooth(\edge) = -1$, $\ideal^\bsmooth_\edge$ is simply defined to be $\PP_{\degreeu\degreev}$.
Similarly, for each interior vertex $\vertex$, we define $\ideal^\bsmooth_\vertex := \sum_{ \edge\ni \vertex} \ideal^\bsmooth_\edge$.

\begin{remark}
	In the above, we have suppressed the dependence of the different vector spaces on $(\degreeu, \degreev)$ to simplify the reading (and writing) of the text.
\end{remark}

\subsection{Topological chain complexes}\label{ss:homological_interpretation}
Any spline $f \in \splSpace^\bsmooth$ is a piecewise polynomial function on $\mesh$.
We can explicitly refer to its piecewise polynomial nature by equivalently expressing it $\sum_{\face} \faceE{} f_\face$ with $f_\face := f|_\face$.
This notation makes it clear that the polynomial $f_\face$ is attached to the face $\face$ of $\mesh$.
Using this notation and Lemma \ref{lem:smoothness}, the spline space $\splSpace^\bsmooth$ can be equivalently expressed as the kernel of the map $\overline{\boundary}$,
\begin{equation*}
	\overline{\boundary} : \moplus_{\face \in \meshF} \faceE{} \PP_{\degreeu\degreev} \rightarrow \moplus_{\edge \in \meshInteriorE} \edgeE{} \PP_{\degreeu\degreev}/\ideal^\bsmooth_\edge\;.
\end{equation*}
defined by composing the boundary map $\boundary$ with the natural quotient map.

As a result of this observation, the spline space $\splSpace^\bsmooth$ can be interpreted as the top homology of a suitably defined chain complex $\quotientComplex^\bsmooth$,
\begin{equation*}
	\begin{tikzcd}
		\quotientComplex^\bsmooth~:~ & \bigoplus\limits_{\face \in \meshF} \faceE{} \PP_{\degreeu\degreev} \arrow[r] & \bigoplus\limits_{\edge \in \meshInteriorE} \edgeE{} \PP_{\degreeu\degreev}/\ideal^\bsmooth_\edge \arrow[r] & \bigoplus\limits_{\vertex \in \meshInteriorV} \vertexE{} \PP_{\degreeu\degreev}/\ideal^\bsmooth_\vertex \arrow[r] & 0\;.
	\end{tikzcd}
\end{equation*}
In other words, we have
\begin{equation*}
	\splSpace^\bsmooth \isomorphic \kerwp{\overline{\boundary}}  = H_2(\quotientComplex^\bsmooth)\;.
\end{equation*}
As in \cite{billera1988homology,schenck1997family,mourrain2014dimension}, we will study $\quotientComplex$ using the following short exact sequence of chain complexes,
\begin{equation}
	\begin{tikzcd}
		~ & ~ & 0 \arrow[d,""] & 0 \arrow[d,""] & ~ \\
		\idealComplex^\bsmooth~:~ & 0\arrow[r] \arrow[d]& \bigoplus\limits_{\edge \in \meshInteriorE} \edgeE{} \ideal^\bsmooth_\edge \arrow[r] \arrow[d] & \bigoplus\limits_{\vertex \in \meshInteriorV} \vertexE{} \ideal^\bsmooth_\vertex \arrow[r] \arrow[d] & 0 \\
		\constantComplex~:~ & \bigoplus\limits_{\face \in \meshF} \faceE{} \PP_{\degreeu\degreev} \arrow[r] \arrow[d] & \bigoplus\limits_{\edge \in \meshInteriorE} \edgeE{} \PP_{\degreeu\degreev} \arrow[r] \arrow[d] & \bigoplus\limits_{\vertex \in \meshInteriorV} \vertexE{} \PP_{\degreeu\degreev} \arrow[r] \arrow[d] & 0 \\
		\quotientComplex^\bsmooth~:~ & \bigoplus\limits_{\face \in \meshF} \faceE{} \PP_{\degreeu\degreev} \arrow[r] & \bigoplus\limits_{\edge \in \meshInteriorE} \edgeE{} \PP_{\degreeu\degreev}/\ideal^\bsmooth_\edge \arrow[r] \arrow[d] & \bigoplus\limits_{\vertex \in \meshInteriorV} \vertexE{} \PP_{\degreeu\degreev}/\ideal^\bsmooth_\vertex \arrow[r] \arrow[d] & 0 \\
		~ & ~ & 0 & 0 & ~
	\end{tikzcd}
\label{eq:complex}
\end{equation}

\begin{theorem} For a simply connected T-mesh $\meshF$, the dimension of the spline space of bi-degree $(m,m')$ and smoothness distribution $\bsmooth$ is given by
	\begin{equation*}
	\begin{split}
		\dimwp{\splSpace^\bsmooth} &= \euler{\quotientComplex^\bsmooth} + \dimwp{H_0(\idealComplex^\bsmooth)}\;,
	\end{split}
	\end{equation*}
where $H_0(\idealComplex^\bsmooth)$ is the zeroth homology of the complex $\idealComplex^\bsmooth$ and $\euler{\quotientComplex^\bsmooth}$ is the Euler characteristic of the complex $\quotientComplex^\bsmooth$,
\begin{equation*}
	\begin{split}
		\euler{\quotientComplex^\bsmooth} &= \ncells{2}(\degreeu+1)(\degreev+1)\\
		&~~- (\degreeu+1)\sum_{\edge \in \meshEH}(\min(\bsmooth(\edge),\degreev)+1) - (\degreev+1)\sum_{\edge \in \meshEV}(\min(\bsmooth(\edge),\degreeu)+1)\\
		&~~+\sum_{\vertex \in \meshV}(\min(\bsmooth_h(\vertex),\degreeu)+1)(\min(\bsmooth_v(\vertex),\degreev)+1)\;.
	\end{split}	
\end{equation*}
\end{theorem}
\begin{proof}
	Following Assumption \ref{ass:simplyConnectedDomain}, it is clear that $H_0(\constantComplex) = 0 = H_1(\constantComplex)$.
	Moreover, from the long exact sequence of homology implied by the short exact sequence of complexes in Equation \eqref{eq:complex}, we obtain
	\begin{equation*}
		H_0(\quotientComplex^\bsmooth) = 0\;, \qquad
		H_0(\idealComplex^\bsmooth) \isomorphic H_1(\quotientComplex^\bsmooth)\;.
	\end{equation*}
	Therefore, the claim follows upon recalling $\splSpace^\bsmooth \isomorphic H_2(\quotientComplex^\bsmooth)$ and the definition of the Euler characteristic of $\quotientComplex^\bsmooth$,
	\begin{equation*}
	\begin{split}
		\euler{\quotientComplex^\bsmooth} &= \dimwp{\quotientComplex^\bsmooth_2} - \dimwp{\quotientComplex^\bsmooth_1} + \dimwp{\quotientComplex^\bsmooth_0}\;,\\
		&= \dimwp{H_2(\quotientComplex^\bsmooth)} - \dimwp{H_1(\quotientComplex^\bsmooth)} + \dimwp{H_0(\quotientComplex^\bsmooth)}\;.
	\end{split}
	\end{equation*}
\end{proof}

\begin{corollary}\label{cor:stable_dimension}
	If $\dimwp{H_0(\idealComplex^\bsmooth)} = 0$, then the dimension is stable and can be computed using the following (combinatorial) formula,
	\begin{equation*}
	\dimwp{\splSpace^\bsmooth} = \euler{\quotientComplex^\bsmooth}\;.
	\end{equation*}
\end{corollary}
	
	\section{Spline space $\splSpace^{\bsmoothr} \supseteq \splSpace^\bsmooth$ of reduced regularity}\label{sec:smoothness_reduction}

In this intermediate section, we will relate the dimension of the spline space $\splSpace^\bsmooth$ to the dimension of a spline space $\splSpace^{\bsmoothr}$ obtained by relaxing the regularity requirements.
That is, for all interior edges $\edge$, it will be assumed that $\bsmoothr(\edge) \leq \bsmooth(\edge)$.
This relationship will be utilized in the next section to present sufficient conditions for the dimension of $\splSpace^{\bsmoothr}$ to be stable.

For the spline space $\splSpace^{\bsmoothr}$, let the first and last chain complexes in Equation \eqref{eq:complex} be denoted by $\idealComplex^{\bsmoothr}$ and $\quotientComplex^{\bsmoothr}$, respectively.
The spline space dimension is therefore given as below,
\begin{equation}
	\dimwp{\splSpace^{\bsmoothr}} = \dimwp{H_2(\quotientComplex^{\bsmoothr})} = \euler{\quotientComplex^{\bsmoothr}} + \dimwp{H_0(\idealComplex^{\bsmoothr})}\;.
\end{equation}
Then, by definition of the smoothness distributions $\bsmooth$ and $\bsmoothr$, we have the following inclusion map from $\idealComplex^{\bsmooth}$ to $\idealComplex^{\bsmoothr}$,
\begin{equation*}
	\idealComplex^\bsmooth \xrightarrow{\iota} \idealComplex^{\bsmoothr}\;.
\end{equation*}

\begin{proposition}\label{prop:alternate_ideal_homology}
	If $H_0(\idealComplex^{\bsmooth}) = 0$, then $H_0(\idealComplex^{\bsmoothr}) \isomorphic H_0(\idealComplex^{\bsmoothr} / \idealComplex^{\bsmooth})$.
\end{proposition}
\begin{proof}
	The claim follows from the following short exact sequence of chain complexes (and the long exact sequence of homology implied by it),
	\begin{equation*}
		\begin{tikzcd}
			0 \arrow[r] & \idealComplex^{\bsmooth} \arrow[r] & \idealComplex^{\bsmoothr} \arrow[r] & \idealComplex^{\bsmoothr} / \idealComplex^{\bsmooth} \arrow[r] & 0\;.
		\end{tikzcd}
	\end{equation*}
\end{proof}

The previous result considerably simplifies the task of identifying when $H_0(\idealComplex^{\bsmoothr})$ will vanish because $H_0(\idealComplex^{\bsmoothr} / \idealComplex^{\bsmooth})$ can be a simpler object to study.
Let ${\mesh}_1^\bsmoothr$ be the set of edges $\edge$ for which $\bsmoothr(\edge) < \bsmooth(\edge)$, and let ${\mesh}_0^\bsmoothr$ be the set of interior vertices of the edges $\tau\in {\mesh}_1^\bsmoothr$.
The following result follows.
\begin{lemma}\label{lem:support}
	The complex $\idealComplex^{\bsmoothr} / \idealComplex^{\bsmooth}$ is supported only on  ${\mesh}_1^\bsmoothr$ and  ${\mesh}_0^\bsmoothr$ .
\end{lemma}
\begin{proof}
	The claim follows from the definition of the complexes $\idealComplex^{\bsmoothr}$ and $\idealComplex^{\bsmooth}$.
	Indeed, if $\bsmoothr(\edge) = \bsmooth(\edge)$, then $\ideal^{\bsmoothr}_\edge = \ideal^{\bsmooth}_\edge$ and the cokernel of $\iota$ is zero on $\edge$; similarly for the vertices.
\end{proof}
	
	\section{Dimension of splines of mixed smoothness}\label{sec:dimension}

This section contains our main results.
Starting from a spline space with stable dimension, we specify sufficient conditions when the dimension can still be computed using Corollary \ref{cor:stable_dimension} after the smoothness requirements are relaxed for a subset of the interior edges.
We first define the weight of a connected union of horizontal or vertical edges.
\begin{definition}[Segment and its weight]
	Let $A \subseteq \meshInteriorE\cup\meshV$ be a finite set of horizontal (resp., vertical) edges $\tau\in\meshInteriorE$ together with their vertices $\gamma\in\tau$, such that $\cup_{\edge \in A}\edge$ is connected and it contains at least one edge.
	Then $A$ will be called a horizontal (resp., vertical) segment.
	Its weight $\omega^\bsmooth(A)$ will be defined as
	\begin{equation*}
		\omega^\bsmooth(A) := 
		\begin{dcases}
			\sum_{\gamma\in A} \big(\degreeu - \bsmooth_h(\vertex)\big)_+\; & \text{if } A \text{ is horizontal}\;,\\
			\sum_{\gamma\in A}  \big(\degreev - \bsmooth_v(\vertex)\big)_+\; & \text{if } A \text{ is vertical}\;.
		\end{dcases}
	\end{equation*}
\end{definition}

\begin{theorem}\label{thm:dimension}
	Let $\bsmooth$ be such that $H_0(\idealComplex^\bsmooth) = 0$ and let $A$ be a segment of the mesh.
	 Consider the space $\splSpace^\bsmoothr$ where the smoothness distribution $\bsmoothr$ is defined as follows for some $r \in \ZZ_{\geq -1}$,
	\[\bsmoothr(\edge)=\begin{cases}
		\bsmooth(\edge)\; & \text{for $\edge \notin A \cap \meshInteriorE$},\\
	\smooth  \leq \bsmooth(\edge)\; & \text{for $\edge \in A \cap \meshInteriorE$},\\
	\end{cases}\]
	If  either one of the following two requirements is satisfied,
	\begin{enumerate}[label=(\alph*)]
		\item $A \subsetneq B$ for some segment $B$,  and $\bsmoothr(\edge) \leq \smooth$ for all $\edge \in B$;
		\item $A$ is horizontal and $\omega^{\bsmoothr}(A) = \geq \degreeu+1$; otherwise, $\omega^\bsmooth(A) \geq \degreev+1$;
	\end{enumerate}
	then $H_0(\idealComplex^{\bsmoothr}) = 0$ and $\dimwp{\splSpace^{\bsmoothr}} = \euler{\quotientComplex^{\bsmoothr}}$.
\end{theorem}
\begin{proof}
	By Lemma \ref{lem:support}, the study of $H_0(\idealComplex^{\bsmoothr})$ reduces to the study of $H_0(\idealComplex^{\bsmoothr} / \idealComplex^{\bsmooth})$ on the segment $A$ --- an essentially one dimensional problem.
	Let us prove the claim for the setting when $A$ satisfies condition (b) above; the proof is much simpler when condition (a) is satisfied.
	Consider then the horizontal segment $A$ as shown below; the proof for vertical segments is analogous. 	$A$ contains the edges $\edge_1, \dots, \edge_k\in\meshInteriorE $ and the vertices $\vertex_0, \vertex_1, \dots, \vertex_k$. By definition $A$ contains at least two different vertices.
	
	
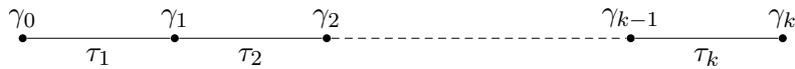
\begin{center}
	\begin{tikzpicture}[scale=2]
		\node[circle,fill=black,black,inner sep=0pt,minimum size=3pt] (c0) at (0,0) {};
		\node[circle,fill=black,black,inner sep=0pt,minimum size=3pt] (c1) at (1,0) {};
		\node[circle,fill=black,black,inner sep=0pt,minimum size=3pt] (c2) at (2,0) {};
		\node[circle,fill=black,black,inner sep=0pt,minimum size=3pt] (c3) at (4,0) {};
		\node[circle,fill=black,black,inner sep=0pt,minimum size=3pt] (c4) at (5,0) {};
		\draw[eThickness] (c0) -- (c1) -- (c2);
		\draw[eThickness, densely dashed] (c2) -- (c3);
		\draw[eThickness] (c3) -- (c4);
		\node[below] at ($(c0)!0.5!(c1)$) {$\edge_1$};
		\node[below] at ($(c1)!0.5!(c2)$) {$\edge_2$};
		\node[below] at ($(c3)!0.5!(c4)$) {$\edge_k$};		
		\node[above] at (c0) {$\vertex_0$};
		\node[above] at (c1) {$\vertex_1$};
		\node[above] at (c2) {$\vertex_2$};
		\node[above] at (c3) {$\vertex_{k-1}$};
		\node[above] at (c4) {$\vertex_k$};
	\end{tikzpicture}
	\end{center}

	Let $\ell_A$ be a non-zero linear polynomial that vanishes on $A$, and let $\ell_0, \dots, \ell_k$ be non-zero linear polynomials that vanish on vertical edges 
	that contain the vertices $\vertex_0, \dots, \vertex_k$, respectively.
	Let $\smooth_0, \dots, \smooth_k$ be such that $\smooth_i = \bsmooth_h(\vertex_i) = \bsmoothr_h(\vertex_i)$.
	
	By definition, $\ideal_\edge^{\bsmoothr} = \bigl\{ \ell_A^{\smooth+1}f \colon f\in \PP_{\degreeu(\degreev-\smooth-1)}\bigr\}$. 
	Since $\omega^{\bsmoothr}(A) \geq \degreeu+1$ and the vertices $\gamma_i$ are all different then, for any  $i\neq j$, there are polynomials $f_i$, $i = 0, \dots, k$, such that  $1=\sum_{i=0}^k\ell_i^{r_i+1}f_i$ \cite[Proposition 1.8]{mourrain2014dimension}.
	Thus we can write
	\begin{equation*}
	\ideal_\edge^{\bsmoothr} 
	= \ell_A^{r+1}\sum_{i=0}^k\ell_{i}^{\smooth_i+1}\PP_{(\degreeu-\smooth_i-1)(\degreev-r-1)}\;.	
	\end{equation*}
	Then, any element $\ell_A^{r+1}f$ in $\ideal^{\bsmoothr}_{\vertex_i}$ can be written as the sum of polynomials  $\ell_A^{r+1}\ell_i^{r_i+1}f_i\in \ideal^{\bsmoothr}_{\vertex_i}$ for some $f_i$ of degree $\leq \degreeu-\smooth_i-1$, $i = 0, \dots, k$.
	But $H_0(\idealComplex^{\bsmooth})=0$ by hypothesis and $\ell_A^{r+1}\ell_i^{r_i+1}f_i\in \ideal^{\bsmooth}_{\vertex_i}$ for all $i$.
	Hence, in the complex $\idealComplex^{\bsmoothr} / \idealComplex^\bsmooth$ all $\vertexE{i} \ideal^{\bsmoothr}_{\vertex_i}/ \ideal^{\bsmooth}_{\vertex_i}$ are in the image of the boundary map.
	Therefore, $H_0(\idealComplex^{\bsmoothr} / \idealComplex^\bsmooth) = 0$ and the claim follows from Proposition \ref{prop:alternate_ideal_homology}.
\end{proof}

\begin{remark}
	Theorem \ref{thm:dimension} discusses the dimension when the smoothness is reduced across a single segment of the mesh.
	Its successive applications can help us compute the dimension of a large class of splines on $\mesh$ with mixed smoothness.
\end{remark}

Let us present an example application of Theorem \ref{thm:dimension} to a special space of splines called PHT-splines \cite{deng2008polynomial}.
Corollary \ref{cor:pht} helps compute the dimension of PHT-splines of mixed smoothness; alternatively, Bernstein--\Bezier techniques can be used to obtain the result.

\begin{definition}[$(\degreeu+1,\degreev+1)$ smoothness distribution]
	The smoothness distribution $\bsmooth$ will be called an $(\degreeu+1,\degreev+1)$ smoothness distribution if for all edges $\edge \in \meshEH$ (resp., $\meshEV$), $\bsmooth(\edge) \leq (\degreev-1)/2$ (resp., $(\degreeu-1)/2$).
\end{definition}

\begin{theorem}\label{thm:pht}
	Let $\bsmooth$ be an $(\degreeu+1,\degreev+1)$ smoothness distribution, and let $\bsmoothr$ be any smoothness distribution such that $\bsmoothr(\edge) \leq \bsmooth(\edge)$ for all edges of $\mesh$.
	If $H_0(\idealComplex^{\bsmooth}) = 0$, then $H_0(\idealComplex^{\bsmoothr}) = 0$ and $\dimwp{\splSpace^{\bsmoothr}} = \euler{\quotientComplex^{\bsmoothr}}$.
\end{theorem}
\begin{proof}
	Since each interior edge is intersected by two transversal edges on its boundary, by the definition of $\bsmooth$ the weight of each interior edge satisfies condition (b) from Theorem \ref{thm:dimension}.
	Therefore, we can move from the smoothness distribution $\bsmooth$ to $\bsmoothr$ one edge at a time; at each stage, $H_0(\idealComplex) = 0$ and Theorem \ref{thm:dimension}(b) will be applicable.
\end{proof}

\begin{corollary}[PHT-splines of mixed smoothness]\label{cor:pht}
	Let $\mesh$ be a hierarchical T-mesh and let $\bsmooth$ be an $(\degreeu+1,\degreev+1)$ smoothness distribution such that $\bsmooth(\edge) = \bsmooth(\edge')$ for all edges $\edge$ and $\edge'$ that belong to the same segment. Then, from \cite{mourrain2014dimension}, $H_0(\idealComplex^{\bsmooth}) = 0$.
	Therefore, for any other smoothness distribution $\bsmoothr$ as in Theorem \ref{thm:pht}, we have
	\begin{equation*}
		\dimwp{\splSpace^{\bsmoothr}} = \euler{\quotientComplex^{\bsmoothr}}\;.
	\end{equation*}
\end{corollary}

	\section{Examples}\label{sec:examples}
\begin{figure}[h]
	\subcaptionbox{$\splSpace^{\bsmooth}$}[0.49\textwidth]{%
	\tikzsetnextfilename{./tikz/images/mixed_reg_ex2a}%
%
%

\begin{tikzpicture}[scale=1.7,transform shape]
	 \begin{scope}
		\coordinate (v0) at (0.0000000000000000,0.0000000000000000) {};
		\coordinate (v1) at (4.0000000000000000,0.0000000000000000) {};
		\coordinate (v2) at (4.0000000000000000,4.0000000000000000) {};
		\coordinate (v3) at (0.0000000000000000,4.0000000000000000) {};
		\coordinate (v4) at (1.0000000000000000,0.0000000000000000) {};
		\coordinate (v5) at (1.0000000000000000,4.0000000000000000) {};
		\coordinate (v6) at (3.0000000000000000,0.0000000000000000) {};
		\coordinate (v7) at (3.0000000000000000,4.0000000000000000) {};
		\coordinate (v8) at (0.0000000000000000,1.0000000000000000) {};
		\coordinate (v9) at (1.0000000000000000,1.0000000000000000) {};
		\coordinate (v10) at (3.0000000000000000,1.0000000000000000) {};
		\coordinate (v11) at (4.0000000000000000,1.0000000000000000) {};
		\coordinate (v12) at (0.0000000000000000,3.0000000000000000) {};
		\coordinate (v13) at (1.0000000000000000,3.0000000000000000) {};
		\coordinate (v14) at (3.0000000000000000,3.0000000000000000) {};
		\coordinate (v15) at (4.0000000000000000,3.0000000000000000) {};
		\coordinate (v16) at (2.0000000000000000,1.0000000000000000) {};
		\coordinate (v17) at (2.0000000000000000,3.0000000000000000) {};
		\coordinate (v18) at (2.0000000000000000,1.5000000000000000) {};
		\coordinate (v19) at (3.0000000000000000,1.5000000000000000) {};
		\coordinate (v20) at (2.0000000000000000,2.5000000000000000) {};
		\coordinate (v21) at (3.0000000000000000,2.5000000000000000) {};
		\coordinate (v22) at (2.5000000000000000,1.5000000000000000) {};
		\coordinate (v23) at (2.5000000000000000,2.5000000000000000) {};
		\coordinate (v24) at (2.0000000000000000,2.0000000000000000) {};
		\coordinate (v25) at (2.5000000000000000,2.0000000000000000) {};

		 \draw[step=1,thin] (v4.center) -- (v9.center) -- (v8.center) -- (v0.center) -- cycle;
		 \draw[step=1,thin] (v6.center) -- (v10.center) -- (v16.center) -- (v9.center) -- (v4.center) -- cycle;
		 \draw[step=1,thin] (v1.center) -- (v11.center) -- (v10.center) -- (v6.center) -- cycle;
		 \draw[step=1,thin] (v9.center) -- (v13.center) -- (v12.center) -- (v8.center) -- cycle;
		 \draw[step=1,thin] (v16.center) -- (v18.center) -- (v24.center) -- (v20.center) -- (v17.center) -- (v13.center) -- (v9.center) -- cycle;
		 \draw[step=1,thin] (v11.center) -- (v15.center) -- (v14.center) -- (v21.center) -- (v19.center) -- (v10.center) -- cycle;
		 \draw[step=1,thin] (v13.center) -- (v5.center) -- (v3.center) -- (v12.center) -- cycle;
		 \draw[step=1,thin] (v17.center) -- (v14.center) -- (v7.center) -- (v5.center) -- (v13.center) -- cycle;
		 \draw[step=1,thin] (v15.center) -- (v2.center) -- (v7.center) -- (v14.center) -- cycle;
		 \draw[step=1,thin] (v10.center) -- (v19.center) -- (v22.center) -- (v18.center) -- (v16.center) -- cycle;
		 \draw[step=1,thin] (v22.center) -- (v25.center) -- (v24.center) -- (v18.center) -- cycle;
		 \draw[step=1,thin] (v23.center) -- (v21.center) -- (v14.center) -- (v17.center) -- (v20.center) -- cycle;
		 \draw[step=1,thin] (v19.center) -- (v21.center) -- (v23.center) -- (v25.center) -- (v22.center) -- cycle;
		 \draw[step=1,thin] (v25.center) -- (v23.center) -- (v20.center) -- (v24.center) -- cycle;

		\node[scale=0.3,draw,black,thin,fill=white,circle,minimum size=2pt,inner sep=0] at (v0) {$\vertex_{0}$};
		\node[scale=0.3,draw,black,thin,fill=white,circle,minimum size=2pt,inner sep=0] at (v1) {$\vertex_{1}$};
		\node[scale=0.3,draw,black,thin,fill=white,circle,minimum size=2pt,inner sep=0] at (v2) {$\vertex_{2}$};
		\node[scale=0.3,draw,black,thin,fill=white,circle,minimum size=2pt,inner sep=0] at (v3) {$\vertex_{3}$};
		\node[scale=0.3,draw,black,thin,fill=white,circle,minimum size=2pt,inner sep=0] at (v4) {$\vertex_{4}$};
		\node[scale=0.3,draw,black,thin,fill=white,circle,minimum size=2pt,inner sep=0] at (v5) {$\vertex_{5}$};
		\node[scale=0.3,draw,black,thin,fill=white,circle,minimum size=2pt,inner sep=0] at (v6) {$\vertex_{6}$};
		\node[scale=0.3,draw,black,thin,fill=white,circle,minimum size=2pt,inner sep=0] at (v7) {$\vertex_{7}$};
		\node[scale=0.3,draw,black,thin,fill=white,circle,minimum size=2pt,inner sep=0] at (v8) {$\vertex_{8}$};
		\node[scale=0.3,draw,black,thin,fill=white,circle,minimum size=2pt,inner sep=0] at (v9) {$\vertex_{9}$};
		\node[scale=0.3,draw,black,thin,fill=white,circle,minimum size=2pt,inner sep=0] at (v10) {$\vertex_{10}$};
		\node[scale=0.3,draw,black,thin,fill=white,circle,minimum size=2pt,inner sep=0] at (v11) {$\vertex_{11}$};
		\node[scale=0.3,draw,black,thin,fill=white,circle,minimum size=2pt,inner sep=0] at (v12) {$\vertex_{12}$};
		\node[scale=0.3,draw,black,thin,fill=white,circle,minimum size=2pt,inner sep=0] at (v13) {$\vertex_{13}$};
		\node[scale=0.3,draw,black,thin,fill=white,circle,minimum size=2pt,inner sep=0] at (v14) {$\vertex_{14}$};
		\node[scale=0.3,draw,black,thin,fill=white,circle,minimum size=2pt,inner sep=0] at (v15) {$\vertex_{15}$};
		\node[scale=0.3,draw,black,thin,fill=white,circle,minimum size=2pt,inner sep=0] at (v16) {$\vertex_{16}$};
		\node[scale=0.3,draw,black,thin,fill=white,circle,minimum size=2pt,inner sep=0] at (v17) {$\vertex_{17}$};
		\node[scale=0.3,draw,black,thin,fill=white,circle,minimum size=2pt,inner sep=0] at (v18) {$\vertex_{18}$};
		\node[scale=0.3,draw,black,thin,fill=white,circle,minimum size=2pt,inner sep=0] at (v19) {$\vertex_{19}$};
		\node[scale=0.3,draw,black,thin,fill=white,circle,minimum size=2pt,inner sep=0] at (v20) {$\vertex_{20}$};
		\node[scale=0.3,draw,black,thin,fill=white,circle,minimum size=2pt,inner sep=0] at (v21) {$\vertex_{21}$};
		\node[scale=0.3,draw,black,thin,fill=white,circle,minimum size=2pt,inner sep=0] at (v22) {$\vertex_{22}$};
		\node[scale=0.3,draw,black,thin,fill=white,circle,minimum size=2pt,inner sep=0] at (v23) {$\vertex_{23}$};
		\node[scale=0.3,draw,black,thin,fill=white,circle,minimum size=2pt,inner sep=0] at (v24) {$\vertex_{24}$};
		\node[scale=0.3,draw,black,thin,fill=white,circle,minimum size=2pt,inner sep=0] at (v25) {$\vertex_{25}$};

		\node[scale=0.3,below,inner sep=0] at ($(v0)!0.5!(v4)$) {$\edge_{0}(-1)$};
		\node[scale=0.3,left,inner sep=0] at ($(v1)!0.5!(v11)$) {$\edge_{1}(-1)$};
		\node[scale=0.3,below,inner sep=0] at ($(v3)!0.5!(v5)$) {$\edge_{2}(-1)$};
		\node[scale=0.3,left,inner sep=0] at ($(v0)!0.5!(v8)$) {$\edge_{3}(-1)$};
		\node[scale=0.3,below,inner sep=0] at ($(v4)!0.5!(v6)$) {$\edge_{4}(-1)$};
		\node[scale=0.3,below,inner sep=0] at ($(v5)!0.5!(v7)$) {$\edge_{5}(-1)$};
		\node[scale=0.3,left,inner sep=0] at ($(v4)!0.5!(v9)$) {$\edge_{6}(1)$};
		\node[scale=0.3,below,inner sep=0] at ($(v6)!0.5!(v1)$) {$\edge_{7}(-1)$};
		\node[scale=0.3,below,inner sep=0] at ($(v7)!0.5!(v2)$) {$\edge_{8}(-1)$};
		\node[scale=0.3,left,inner sep=0] at ($(v6)!0.5!(v10)$) {$\edge_{9}(1)$};
		\node[scale=0.3,left,inner sep=0] at ($(v8)!0.5!(v12)$) {$\edge_{10}(-1)$};
		\node[scale=0.3,left,inner sep=0] at ($(v9)!0.5!(v13)$) {$\edge_{11}(1)$};
		\node[scale=0.3,below,inner sep=0] at ($(v8)!0.5!(v9)$) {$\edge_{12}(1)$};
		\node[scale=0.3,left,inner sep=0] at ($(v10)!0.5!(v19)$) {$\edge_{13}(1)$};
		\node[scale=0.3,below,inner sep=0] at ($(v9)!0.5!(v16)$) {$\edge_{14}(1)$};
		\node[scale=0.3,left,inner sep=0] at ($(v11)!0.5!(v15)$) {$\edge_{15}(-1)$};
		\node[scale=0.3,below,inner sep=0] at ($(v10)!0.5!(v11)$) {$\edge_{16}(1)$};
		\node[scale=0.3,left,inner sep=0] at ($(v12)!0.5!(v3)$) {$\edge_{17}(-1)$};
		\node[scale=0.3,left,inner sep=0] at ($(v13)!0.5!(v5)$) {$\edge_{18}(1)$};
		\node[scale=0.3,below,inner sep=0] at ($(v12)!0.5!(v13)$) {$\edge_{19}(1)$};
		\node[scale=0.3,left,inner sep=0] at ($(v14)!0.5!(v7)$) {$\edge_{20}(1)$};
		\node[scale=0.3,below,inner sep=0] at ($(v13)!0.5!(v17)$) {$\edge_{21}(1)$};
		\node[scale=0.3,left,inner sep=0] at ($(v15)!0.5!(v2)$) {$\edge_{22}(-1)$};
		\node[scale=0.3,below,inner sep=0] at ($(v14)!0.5!(v15)$) {$\edge_{23}(1)$};
		\node[scale=0.3,below,inner sep=0] at ($(v16)!0.5!(v10)$) {$\edge_{24}(1)$};
		\node[scale=0.3,below,inner sep=0] at ($(v17)!0.5!(v14)$) {$\edge_{25}(1)$};
		\node[scale=0.3,left,inner sep=0] at ($(v16)!0.5!(v18)$) {$\edge_{26}(1)$};
		\node[scale=0.3,left,inner sep=0] at ($(v18)!0.5!(v24)$) {$\edge_{27}(1)$};
		\node[scale=0.3,left,inner sep=0] at ($(v19)!0.5!(v21)$) {$\edge_{28}(1)$};
		\node[scale=0.3,below,inner sep=0] at ($(v18)!0.5!(v22)$) {$\edge_{29}(1)$};
		\node[scale=0.3,left,inner sep=0] at ($(v20)!0.5!(v17)$) {$\edge_{30}(1)$};
		\node[scale=0.3,left,inner sep=0] at ($(v21)!0.5!(v14)$) {$\edge_{31}(1)$};
		\node[scale=0.3,below,inner sep=0] at ($(v20)!0.5!(v23)$) {$\edge_{32}(1)$};
		\node[scale=0.3,below,inner sep=0] at ($(v22)!0.5!(v19)$) {$\edge_{33}(1)$};
		\node[scale=0.3,below,inner sep=0] at ($(v23)!0.5!(v21)$) {$\edge_{34}(1)$};
		\node[scale=0.3,left,inner sep=0] at ($(v22)!0.5!(v25)$) {$\edge_{35}(1)$};
		\node[scale=0.3,left,inner sep=0] at ($(v24)!0.5!(v20)$) {$\edge_{36}(1)$};
		\node[scale=0.3,left,inner sep=0] at ($(v25)!0.5!(v23)$) {$\edge_{37}(1)$};
		\node[scale=0.3,below,inner sep=0] at ($(v24)!0.5!(v25)$) {$\edge_{38}(1)$};

	 \end{scope}
\end{tikzpicture}

%
}
	\subcaptionbox{$\splSpace^{\bsmoothr}$}[0.49\textwidth]{%
	\tikzsetnextfilename{./tikz/images/mixed_reg_ex2b}%
%
%

\begin{tikzpicture}[scale=1.7,transform shape]
	 \begin{scope}
		\coordinate (v0) at (0.0000000000000000,0.0000000000000000) {};
		\coordinate (v1) at (4.0000000000000000,0.0000000000000000) {};
		\coordinate (v2) at (4.0000000000000000,4.0000000000000000) {};
		\coordinate (v3) at (0.0000000000000000,4.0000000000000000) {};
		\coordinate (v4) at (1.0000000000000000,0.0000000000000000) {};
		\coordinate (v5) at (1.0000000000000000,4.0000000000000000) {};
		\coordinate (v6) at (3.0000000000000000,0.0000000000000000) {};
		\coordinate (v7) at (3.0000000000000000,4.0000000000000000) {};
		\coordinate (v8) at (0.0000000000000000,1.0000000000000000) {};
		\coordinate (v9) at (1.0000000000000000,1.0000000000000000) {};
		\coordinate (v10) at (3.0000000000000000,1.0000000000000000) {};
		\coordinate (v11) at (4.0000000000000000,1.0000000000000000) {};
		\coordinate (v12) at (0.0000000000000000,3.0000000000000000) {};
		\coordinate (v13) at (1.0000000000000000,3.0000000000000000) {};
		\coordinate (v14) at (3.0000000000000000,3.0000000000000000) {};
		\coordinate (v15) at (4.0000000000000000,3.0000000000000000) {};
		\coordinate (v16) at (2.0000000000000000,1.0000000000000000) {};
		\coordinate (v17) at (2.0000000000000000,3.0000000000000000) {};
		\coordinate (v18) at (2.0000000000000000,1.5000000000000000) {};
		\coordinate (v19) at (3.0000000000000000,1.5000000000000000) {};
		\coordinate (v20) at (2.0000000000000000,2.5000000000000000) {};
		\coordinate (v21) at (3.0000000000000000,2.5000000000000000) {};
		\coordinate (v22) at (2.5000000000000000,1.5000000000000000) {};
		\coordinate (v23) at (2.5000000000000000,2.5000000000000000) {};
		\coordinate (v24) at (2.0000000000000000,2.0000000000000000) {};
		\coordinate (v25) at (2.5000000000000000,2.0000000000000000) {};

		 \draw[step=1,thin] (v4.center) -- (v9.center) -- (v8.center) -- (v0.center) -- cycle;
		 \draw[step=1,thin] (v6.center) -- (v10.center) -- (v16.center) -- (v9.center) -- (v4.center) -- cycle;
		 \draw[step=1,thin] (v1.center) -- (v11.center) -- (v10.center) -- (v6.center) -- cycle;
		 \draw[step=1,thin] (v9.center) -- (v13.center) -- (v12.center) -- (v8.center) -- cycle;
		 \draw[step=1,thin] (v16.center) -- (v18.center) -- (v24.center) -- (v20.center) -- (v17.center) -- (v13.center) -- (v9.center) -- cycle;
		 \draw[step=1,thin] (v11.center) -- (v15.center) -- (v14.center) -- (v21.center) -- (v19.center) -- (v10.center) -- cycle;
		 \draw[step=1,thin] (v13.center) -- (v5.center) -- (v3.center) -- (v12.center) -- cycle;
		 \draw[step=1,thin] (v17.center) -- (v14.center) -- (v7.center) -- (v5.center) -- (v13.center) -- cycle;
		 \draw[step=1,thin] (v15.center) -- (v2.center) -- (v7.center) -- (v14.center) -- cycle;
		 \draw[step=1,thin] (v10.center) -- (v19.center) -- (v22.center) -- (v18.center) -- (v16.center) -- cycle;
		 \draw[step=1,thin] (v22.center) -- (v25.center) -- (v24.center) -- (v18.center) -- cycle;
		 \draw[step=1,thin] (v23.center) -- (v21.center) -- (v14.center) -- (v17.center) -- (v20.center) -- cycle;
		 \draw[step=1,thin] (v19.center) -- (v21.center) -- (v23.center) -- (v25.center) -- (v22.center) -- cycle;
		 \draw[step=1,thin] (v25.center) -- (v23.center) -- (v20.center) -- (v24.center) -- cycle;

		\node[scale=0.3,draw,black,thin,fill=white,circle,minimum size=2pt,inner sep=0] at (v0) {$\vertex_{0}$};
		\node[scale=0.3,draw,black,thin,fill=white,circle,minimum size=2pt,inner sep=0] at (v1) {$\vertex_{1}$};
		\node[scale=0.3,draw,black,thin,fill=white,circle,minimum size=2pt,inner sep=0] at (v2) {$\vertex_{2}$};
		\node[scale=0.3,draw,black,thin,fill=white,circle,minimum size=2pt,inner sep=0] at (v3) {$\vertex_{3}$};
		\node[scale=0.3,draw,black,thin,fill=white,circle,minimum size=2pt,inner sep=0] at (v4) {$\vertex_{4}$};
		\node[scale=0.3,draw,black,thin,fill=white,circle,minimum size=2pt,inner sep=0] at (v5) {$\vertex_{5}$};
		\node[scale=0.3,draw,black,thin,fill=white,circle,minimum size=2pt,inner sep=0] at (v6) {$\vertex_{6}$};
		\node[scale=0.3,draw,black,thin,fill=white,circle,minimum size=2pt,inner sep=0] at (v7) {$\vertex_{7}$};
		\node[scale=0.3,draw,black,thin,fill=white,circle,minimum size=2pt,inner sep=0] at (v8) {$\vertex_{8}$};
		\node[scale=0.3,draw,black,thin,fill=white,circle,minimum size=2pt,inner sep=0] at (v9) {$\vertex_{9}$};
		\node[scale=0.3,draw,black,thin,fill=white,circle,minimum size=2pt,inner sep=0] at (v10) {$\vertex_{10}$};
		\node[scale=0.3,draw,black,thin,fill=white,circle,minimum size=2pt,inner sep=0] at (v11) {$\vertex_{11}$};
		\node[scale=0.3,draw,black,thin,fill=white,circle,minimum size=2pt,inner sep=0] at (v12) {$\vertex_{12}$};
		\node[scale=0.3,draw,black,thin,fill=white,circle,minimum size=2pt,inner sep=0] at (v13) {$\vertex_{13}$};
		\node[scale=0.3,draw,black,thin,fill=white,circle,minimum size=2pt,inner sep=0] at (v14) {$\vertex_{14}$};
		\node[scale=0.3,draw,black,thin,fill=white,circle,minimum size=2pt,inner sep=0] at (v15) {$\vertex_{15}$};
		\node[scale=0.3,draw,black,thin,fill=white,circle,minimum size=2pt,inner sep=0] at (v16) {$\vertex_{16}$};
		\node[scale=0.3,draw,black,thin,fill=white,circle,minimum size=2pt,inner sep=0] at (v17) {$\vertex_{17}$};
		\node[scale=0.3,draw,black,thin,fill=white,circle,minimum size=2pt,inner sep=0] at (v18) {$\vertex_{18}$};
		\node[scale=0.3,draw,black,thin,fill=white,circle,minimum size=2pt,inner sep=0] at (v19) {$\vertex_{19}$};
		\node[scale=0.3,draw,black,thin,fill=white,circle,minimum size=2pt,inner sep=0] at (v20) {$\vertex_{20}$};
		\node[scale=0.3,draw,black,thin,fill=white,circle,minimum size=2pt,inner sep=0] at (v21) {$\vertex_{21}$};
		\node[scale=0.3,draw,black,thin,fill=white,circle,minimum size=2pt,inner sep=0] at (v22) {$\vertex_{22}$};
		\node[scale=0.3,draw,black,thin,fill=white,circle,minimum size=2pt,inner sep=0] at (v23) {$\vertex_{23}$};
		\node[scale=0.3,draw,black,thin,fill=white,circle,minimum size=2pt,inner sep=0] at (v24) {$\vertex_{24}$};
		\node[scale=0.3,draw,black,thin,fill=white,circle,minimum size=2pt,inner sep=0] at (v25) {$\vertex_{25}$};

		\node[scale=0.3,below,inner sep=0] at ($(v0)!0.5!(v4)$) {$\edge_{0}(-1)$};
		\node[scale=0.3,left,inner sep=0] at ($(v1)!0.5!(v11)$) {$\edge_{1}(-1)$};
		\node[scale=0.3,below,inner sep=0] at ($(v3)!0.5!(v5)$) {$\edge_{2}(-1)$};
		\node[scale=0.3,left,inner sep=0] at ($(v0)!0.5!(v8)$) {$\edge_{3}(-1)$};
		\node[scale=0.3,below,inner sep=0] at ($(v4)!0.5!(v6)$) {$\edge_{4}(-1)$};
		\node[scale=0.3,below,inner sep=0] at ($(v5)!0.5!(v7)$) {$\edge_{5}(-1)$};
		\node[scale=0.3,left,inner sep=0] at ($(v4)!0.5!(v9)$) {$\edge_{6}(1)$};
		\node[scale=0.3,below,inner sep=0] at ($(v6)!0.5!(v1)$) {$\edge_{7}(-1)$};
		\node[scale=0.3,below,inner sep=0] at ($(v7)!0.5!(v2)$) {$\edge_{8}(-1)$};
		\node[scale=0.3,left,inner sep=0] at ($(v6)!0.5!(v10)$) {$\edge_{9}(1)$};
		\node[scale=0.3,left,inner sep=0] at ($(v8)!0.5!(v12)$) {$\edge_{10}(-1)$};
		\node[scale=0.3,left,inner sep=0] at ($(v9)!0.5!(v13)$) {$\edge_{11}(1)$};
		\node[scale=0.3,below,inner sep=0] at ($(v8)!0.5!(v9)$) {$\edge_{12}(1)$};
		\node[scale=0.3,left,inner sep=0] at ($(v10)!0.5!(v19)$) {$\edge_{13}(1)$};
		\node[scale=0.3,below,inner sep=0] at ($(v9)!0.5!(v16)$) {$\edge_{14}(1)$};
		\node[scale=0.3,left,inner sep=0] at ($(v11)!0.5!(v15)$) {$\edge_{15}(-1)$};
		\node[scale=0.3,below,inner sep=0] at ($(v10)!0.5!(v11)$) {$\edge_{16}(1)$};
		\node[scale=0.3,left,inner sep=0] at ($(v12)!0.5!(v3)$) {$\edge_{17}(-1)$};
		\node[scale=0.3,left,inner sep=0] at ($(v13)!0.5!(v5)$) {$\edge_{18}(1)$};
		\node[scale=0.3,below,inner sep=0] at ($(v12)!0.5!(v13)$) {$\edge_{19}(1)$};
		\node[scale=0.3,left,inner sep=0] at ($(v14)!0.5!(v7)$) {$\edge_{20}(1)$};
		\node[scale=0.3,below,inner sep=0] at ($(v13)!0.5!(v17)$) {$\edge_{21}(1)$};
		\node[scale=0.3,left,inner sep=0] at ($(v15)!0.5!(v2)$) {$\edge_{22}(-1)$};
		\node[scale=0.3,below,inner sep=0] at ($(v14)!0.5!(v15)$) {$\edge_{23}(1)$};
		\node[scale=0.3,below,inner sep=0] at ($(v16)!0.5!(v10)$) {$\edge_{24}(1)$};
		\node[scale=0.3,below,inner sep=0] at ($(v17)!0.5!(v14)$) {$\edge_{25}(1)$};
		\node[scale=0.3,left,inner sep=0] at ($(v16)!0.5!(v18)$) {$\edge_{26}(1)$};
		\node[scale=0.3,left,inner sep=0] at ($(v18)!0.5!(v24)$) {$\color{blue}\edge_{27}(0)$};
		\node[scale=0.3,left,inner sep=0] at ($(v19)!0.5!(v21)$) {$\edge_{28}(1)$};
		\node[scale=0.3,below,inner sep=0] at ($(v18)!0.5!(v22)$) {$\edge_{29}(1)$};
		\node[scale=0.3,left,inner sep=0] at ($(v20)!0.5!(v17)$) {$\edge_{30}(1)$};
		\node[scale=0.3,left,inner sep=0] at ($(v21)!0.5!(v14)$) {$\edge_{31}(1)$};
		\node[scale=0.3,below,inner sep=0] at ($(v20)!0.5!(v23)$) {$\edge_{32}(1)$};
		\node[scale=0.3,below,inner sep=0] at ($(v22)!0.5!(v19)$) {$\edge_{33}(1)$};
		\node[scale=0.3,below,inner sep=0] at ($(v23)!0.5!(v21)$) {$\color{blue}\edge_{34}(-1)$};
		\node[scale=0.3,left,inner sep=0] at ($(v22)!0.5!(v25)$) {$\color{blue}\edge_{35}(0)$};
		\node[scale=0.3,left,inner sep=0] at ($(v24)!0.5!(v20)$) {$\color{blue}\edge_{36}(0)$};
		\node[scale=0.3,left,inner sep=0] at ($(v25)!0.5!(v23)$) {$\edge_{37}(1)$};
		\node[scale=0.3,below,inner sep=0] at ($(v24)!0.5!(v25)$) {$\edge_{38}(1)$};

	 \end{scope}
\end{tikzpicture}

%
}
	\caption{The above figures correspond to the PHT-spline setting considered in Example \ref{ex:ex1}.
		The smoothness required across each edge has been annotated in parenthesis next to the edge labels.
		Figure (a) shows the initial smoothness distribution, while Figure (b) shows the modified initial distributions; the modifications are limited to the edges labelled in blue.
	}
	\label{fig:ex1}
\end{figure}

\begin{figure}[p]
	\subcaptionbox{$\splSpace^{\bsmooth}$}[0.49\textwidth]{%
	\tikzsetnextfilename{./tikz/images/mixed_reg_ex1a}%
%
%

\begin{tikzpicture}[scale=1.7, transform shape]
	 \begin{scope}
		\coordinate (v0) at (0.0000000000000000,0.0000000000000000) {};
		\coordinate (v1) at (4.0000000000000000,0.0000000000000000) {};
		\coordinate (v2) at (4.0000000000000000,4.0000000000000000) {};
		\coordinate (v3) at (0.0000000000000000,4.0000000000000000) {};
		\coordinate (v4) at (1.0000000000000000,0.0000000000000000) {};
		\coordinate (v5) at (1.0000000000000000,4.0000000000000000) {};
		\coordinate (v6) at (3.0000000000000000,0.0000000000000000) {};
		\coordinate (v7) at (3.0000000000000000,4.0000000000000000) {};
		\coordinate (v8) at (0.0000000000000000,1.0000000000000000) {};
		\coordinate (v9) at (1.0000000000000000,1.0000000000000000) {};
		\coordinate (v10) at (3.0000000000000000,1.0000000000000000) {};
		\coordinate (v11) at (4.0000000000000000,1.0000000000000000) {};
		\coordinate (v12) at (0.0000000000000000,3.0000000000000000) {};
		\coordinate (v13) at (1.0000000000000000,3.0000000000000000) {};
		\coordinate (v14) at (3.0000000000000000,3.0000000000000000) {};
		\coordinate (v15) at (4.0000000000000000,3.0000000000000000) {};
		\coordinate (v16) at (2.0000000000000000,1.0000000000000000) {};
		\coordinate (v17) at (2.0000000000000000,3.0000000000000000) {};
		\coordinate (v18) at (2.0000000000000000,1.5000000000000000) {};
		\coordinate (v19) at (3.0000000000000000,1.5000000000000000) {};
		\coordinate (v20) at (2.0000000000000000,2.5000000000000000) {};
		\coordinate (v21) at (3.0000000000000000,2.5000000000000000) {};
		\coordinate (v22) at (2.5000000000000000,1.5000000000000000) {};
		\coordinate (v23) at (2.5000000000000000,2.5000000000000000) {};
		\coordinate (v24) at (1.0000000000000000,2.0000000000000000) {};
		\coordinate (v25) at (2.0000000000000000,2.0000000000000000) {};
		\coordinate (v26) at (2.5000000000000000,2.0000000000000000) {};

		 \draw[step=1,thin] (v4.center) -- (v9.center) -- (v8.center) -- (v0.center) -- cycle;
		 \draw[step=1,thin] (v6.center) -- (v10.center) -- (v16.center) -- (v9.center) -- (v4.center) -- cycle;
		 \draw[step=1,thin] (v1.center) -- (v11.center) -- (v10.center) -- (v6.center) -- cycle;
		 \draw[step=1,thin] (v9.center) -- (v24.center) -- (v13.center) -- (v12.center) -- (v8.center) -- cycle;
		 \draw[step=1,thin] (v16.center) -- (v18.center) -- (v25.center) -- (v24.center) -- (v9.center) -- cycle;
		 \draw[step=1,thin] (v11.center) -- (v15.center) -- (v14.center) -- (v21.center) -- (v19.center) -- (v10.center) -- cycle;
		 \draw[step=1,thin] (v13.center) -- (v5.center) -- (v3.center) -- (v12.center) -- cycle;
		 \draw[step=1,thin] (v17.center) -- (v14.center) -- (v7.center) -- (v5.center) -- (v13.center) -- cycle;
		 \draw[step=1,thin] (v15.center) -- (v2.center) -- (v7.center) -- (v14.center) -- cycle;
		 \draw[step=1,thin] (v10.center) -- (v19.center) -- (v22.center) -- (v18.center) -- (v16.center) -- cycle;
		 \draw[step=1,thin] (v22.center) -- (v26.center) -- (v25.center) -- (v18.center) -- cycle;
		 \draw[step=1,thin] (v23.center) -- (v21.center) -- (v14.center) -- (v17.center) -- (v20.center) -- cycle;
		 \draw[step=1,thin] (v19.center) -- (v21.center) -- (v23.center) -- (v26.center) -- (v22.center) -- cycle;
		 \draw[step=1,thin] (v25.center) -- (v20.center) -- (v17.center) -- (v13.center) -- (v24.center) -- cycle;
		 \draw[step=1,thin] (v26.center) -- (v23.center) -- (v20.center) -- (v25.center) -- cycle;

		\node[scale=0.3,draw,black,thin,fill=white,circle,minimum size=2pt,inner sep=0] at (v0) {$\vertex_{0}$};
		\node[scale=0.3,draw,black,thin,fill=white,circle,minimum size=2pt,inner sep=0] at (v1) {$\vertex_{1}$};
		\node[scale=0.3,draw,black,thin,fill=white,circle,minimum size=2pt,inner sep=0] at (v2) {$\vertex_{2}$};
		\node[scale=0.3,draw,black,thin,fill=white,circle,minimum size=2pt,inner sep=0] at (v3) {$\vertex_{3}$};
		\node[scale=0.3,draw,black,thin,fill=white,circle,minimum size=2pt,inner sep=0] at (v4) {$\vertex_{4}$};
		\node[scale=0.3,draw,black,thin,fill=white,circle,minimum size=2pt,inner sep=0] at (v5) {$\vertex_{5}$};
		\node[scale=0.3,draw,black,thin,fill=white,circle,minimum size=2pt,inner sep=0] at (v6) {$\vertex_{6}$};
		\node[scale=0.3,draw,black,thin,fill=white,circle,minimum size=2pt,inner sep=0] at (v7) {$\vertex_{7}$};
		\node[scale=0.3,draw,black,thin,fill=white,circle,minimum size=2pt,inner sep=0] at (v8) {$\vertex_{8}$};
		\node[scale=0.3,draw,black,thin,fill=white,circle,minimum size=2pt,inner sep=0] at (v9) {$\vertex_{9}$};
		\node[scale=0.3,draw,black,thin,fill=white,circle,minimum size=2pt,inner sep=0] at (v10) {$\vertex_{10}$};
		\node[scale=0.3,draw,black,thin,fill=white,circle,minimum size=2pt,inner sep=0] at (v11) {$\vertex_{11}$};
		\node[scale=0.3,draw,black,thin,fill=white,circle,minimum size=2pt,inner sep=0] at (v12) {$\vertex_{12}$};
		\node[scale=0.3,draw,black,thin,fill=white,circle,minimum size=2pt,inner sep=0] at (v13) {$\vertex_{13}$};
		\node[scale=0.3,draw,black,thin,fill=white,circle,minimum size=2pt,inner sep=0] at (v14) {$\vertex_{14}$};
		\node[scale=0.3,draw,black,thin,fill=white,circle,minimum size=2pt,inner sep=0] at (v15) {$\vertex_{15}$};
		\node[scale=0.3,draw,black,thin,fill=white,circle,minimum size=2pt,inner sep=0] at (v16) {$\vertex_{16}$};
		\node[scale=0.3,draw,black,thin,fill=white,circle,minimum size=2pt,inner sep=0] at (v17) {$\vertex_{17}$};
		\node[scale=0.3,draw,black,thin,fill=white,circle,minimum size=2pt,inner sep=0] at (v18) {$\vertex_{18}$};
		\node[scale=0.3,draw,black,thin,fill=white,circle,minimum size=2pt,inner sep=0] at (v19) {$\vertex_{19}$};
		\node[scale=0.3,draw,black,thin,fill=white,circle,minimum size=2pt,inner sep=0] at (v20) {$\vertex_{20}$};
		\node[scale=0.3,draw,black,thin,fill=white,circle,minimum size=2pt,inner sep=0] at (v21) {$\vertex_{21}$};
		\node[scale=0.3,draw,black,thin,fill=white,circle,minimum size=2pt,inner sep=0] at (v22) {$\vertex_{22}$};
		\node[scale=0.3,draw,black,thin,fill=white,circle,minimum size=2pt,inner sep=0] at (v23) {$\vertex_{23}$};
		\node[scale=0.3,draw,black,thin,fill=white,circle,minimum size=2pt,inner sep=0] at (v24) {$\vertex_{24}$};
		\node[scale=0.3,draw,black,thin,fill=white,circle,minimum size=2pt,inner sep=0] at (v25) {$\vertex_{25}$};
		\node[scale=0.3,draw,black,thin,fill=white,circle,minimum size=2pt,inner sep=0] at (v26) {$\vertex_{26}$};

		\node[scale=0.3,below,inner sep=0] at ($(v0)!0.5!(v4)$) {$\edge_{0}(-1)$};
		\node[scale=0.3,left,inner sep=0] at ($(v1)!0.5!(v11)$) {$\edge_{1}(-1)$};
		\node[scale=0.3,below,inner sep=0] at ($(v3)!0.5!(v5)$) {$\edge_{2}(-1)$};
		\node[scale=0.3,left,inner sep=0] at ($(v0)!0.5!(v8)$) {$\edge_{3}(-1)$};
		\node[scale=0.3,below,inner sep=0] at ($(v4)!0.5!(v6)$) {$\edge_{4}(-1)$};
		\node[scale=0.3,below,inner sep=0] at ($(v5)!0.5!(v7)$) {$\edge_{5}(-1)$};
		\node[scale=0.3,left,inner sep=0] at ($(v4)!0.5!(v9)$) {$\edge_{6}(2)$};
		\node[scale=0.3,below,inner sep=0] at ($(v6)!0.5!(v1)$) {$\edge_{7}(-1)$};
		\node[scale=0.3,below,inner sep=0] at ($(v7)!0.5!(v2)$) {$\edge_{8}(-1)$};
		\node[scale=0.3,left,inner sep=0] at ($(v6)!0.5!(v10)$) {$\edge_{9}(1)$};
		\node[scale=0.3,left,inner sep=0] at ($(v8)!0.5!(v12)$) {$\edge_{10}(-1)$};
		\node[scale=0.3,left,inner sep=0] at ($(v9)!0.5!(v24)$) {$\edge_{11}(2)$};
		\node[scale=0.3,below,inner sep=0] at ($(v8)!0.5!(v9)$) {$\edge_{12}(0)$};
		\node[scale=0.3,left,inner sep=0] at ($(v10)!0.5!(v19)$) {$\edge_{13}(1)$};
		\node[scale=0.3,below,inner sep=0] at ($(v9)!0.5!(v16)$) {$\edge_{14}(0)$};
		\node[scale=0.3,left,inner sep=0] at ($(v11)!0.5!(v15)$) {$\edge_{15}(-1)$};
		\node[scale=0.3,below,inner sep=0] at ($(v10)!0.5!(v11)$) {$\edge_{16}(0)$};
		\node[scale=0.3,left,inner sep=0] at ($(v12)!0.5!(v3)$) {$\edge_{17}(-1)$};
		\node[scale=0.3,left,inner sep=0] at ($(v13)!0.5!(v5)$) {$\edge_{18}(2)$};
		\node[scale=0.3,below,inner sep=0] at ($(v12)!0.5!(v13)$) {$\edge_{19}(2)$};
		\node[scale=0.3,left,inner sep=0] at ($(v14)!0.5!(v7)$) {$\edge_{20}(1)$};
		\node[scale=0.3,below,inner sep=0] at ($(v13)!0.5!(v17)$) {$\edge_{21}(2)$};
		\node[scale=0.3,left,inner sep=0] at ($(v15)!0.5!(v2)$) {$\edge_{22}(-1)$};
		\node[scale=0.3,below,inner sep=0] at ($(v14)!0.5!(v15)$) {$\edge_{23}(2)$};
		\node[scale=0.3,below,inner sep=0] at ($(v16)!0.5!(v10)$) {$\edge_{24}(0)$};
		\node[scale=0.3,below,inner sep=0] at ($(v17)!0.5!(v14)$) {$\edge_{25}(2)$};
		\node[scale=0.3,left,inner sep=0] at ($(v16)!0.5!(v18)$) {$\edge_{26}(1)$};
		\node[scale=0.3,left,inner sep=0] at ($(v18)!0.5!(v25)$) {$\edge_{27}(1)$};
		\node[scale=0.3,left,inner sep=0] at ($(v19)!0.5!(v21)$) {$\edge_{28}(1)$};
		\node[scale=0.3,below,inner sep=0] at ($(v18)!0.5!(v22)$) {$\edge_{29}(0)$};
		\node[scale=0.3,left,inner sep=0] at ($(v20)!0.5!(v17)$) {$\edge_{30}(1)$};
		\node[scale=0.3,left,inner sep=0] at ($(v21)!0.5!(v14)$) {$\edge_{31}(1)$};
		\node[scale=0.3,below,inner sep=0] at ($(v20)!0.5!(v23)$) {$\edge_{32}(2)$};
		\node[scale=0.3,below,inner sep=0] at ($(v22)!0.5!(v19)$) {$\edge_{33}(0)$};
		\node[scale=0.3,below,inner sep=0] at ($(v23)!0.5!(v21)$) {$\edge_{34}(2)$};
		\node[scale=0.3,left,inner sep=0] at ($(v22)!0.5!(v26)$) {$\edge_{35}(2)$};
		\node[scale=0.3,left,inner sep=0] at ($(v24)!0.5!(v13)$) {$\edge_{36}(2)$};
		\node[scale=0.3,left,inner sep=0] at ($(v25)!0.5!(v20)$) {$\edge_{37}(1)$};
		\node[scale=0.3,below,inner sep=0] at ($(v24)!0.5!(v25)$) {$\edge_{38}(2)$};
		\node[scale=0.3,left,inner sep=0] at ($(v26)!0.5!(v23)$) {$\edge_{39}(2)$};
		\node[scale=0.3,below,inner sep=0] at ($(v25)!0.5!(v26)$) {$\edge_{40}(2)$};

	 \end{scope}
\end{tikzpicture}

%
}
	\subcaptionbox{$\splSpace^{\bsmoothr_0}$}[0.49\textwidth]{%
	\tikzsetnextfilename{./tikz/images/mixed_reg_ex1b}%
%
%

\begin{tikzpicture}[scale=1.7, transform shape]
	 \begin{scope}
		\coordinate (v0) at (0.0000000000000000,0.0000000000000000) {};
		\coordinate (v1) at (4.0000000000000000,0.0000000000000000) {};
		\coordinate (v2) at (4.0000000000000000,4.0000000000000000) {};
		\coordinate (v3) at (0.0000000000000000,4.0000000000000000) {};
		\coordinate (v4) at (1.0000000000000000,0.0000000000000000) {};
		\coordinate (v5) at (1.0000000000000000,4.0000000000000000) {};
		\coordinate (v6) at (3.0000000000000000,0.0000000000000000) {};
		\coordinate (v7) at (3.0000000000000000,4.0000000000000000) {};
		\coordinate (v8) at (0.0000000000000000,1.0000000000000000) {};
		\coordinate (v9) at (1.0000000000000000,1.0000000000000000) {};
		\coordinate (v10) at (3.0000000000000000,1.0000000000000000) {};
		\coordinate (v11) at (4.0000000000000000,1.0000000000000000) {};
		\coordinate (v12) at (0.0000000000000000,3.0000000000000000) {};
		\coordinate (v13) at (1.0000000000000000,3.0000000000000000) {};
		\coordinate (v14) at (3.0000000000000000,3.0000000000000000) {};
		\coordinate (v15) at (4.0000000000000000,3.0000000000000000) {};
		\coordinate (v16) at (2.0000000000000000,1.0000000000000000) {};
		\coordinate (v17) at (2.0000000000000000,3.0000000000000000) {};
		\coordinate (v18) at (2.0000000000000000,1.5000000000000000) {};
		\coordinate (v19) at (3.0000000000000000,1.5000000000000000) {};
		\coordinate (v20) at (2.0000000000000000,2.5000000000000000) {};
		\coordinate (v21) at (3.0000000000000000,2.5000000000000000) {};
		\coordinate (v22) at (2.5000000000000000,1.5000000000000000) {};
		\coordinate (v23) at (2.5000000000000000,2.5000000000000000) {};
		\coordinate (v24) at (1.0000000000000000,2.0000000000000000) {};
		\coordinate (v25) at (2.0000000000000000,2.0000000000000000) {};
		\coordinate (v26) at (2.5000000000000000,2.0000000000000000) {};

		 \draw[step=1,thin] (v4.center) -- (v9.center) -- (v8.center) -- (v0.center) -- cycle;
		 \draw[step=1,thin] (v6.center) -- (v10.center) -- (v16.center) -- (v9.center) -- (v4.center) -- cycle;
		 \draw[step=1,thin] (v1.center) -- (v11.center) -- (v10.center) -- (v6.center) -- cycle;
		 \draw[step=1,thin] (v9.center) -- (v24.center) -- (v13.center) -- (v12.center) -- (v8.center) -- cycle;
		 \draw[step=1,thin] (v16.center) -- (v18.center) -- (v25.center) -- (v24.center) -- (v9.center) -- cycle;
		 \draw[step=1,thin] (v11.center) -- (v15.center) -- (v14.center) -- (v21.center) -- (v19.center) -- (v10.center) -- cycle;
		 \draw[step=1,thin] (v13.center) -- (v5.center) -- (v3.center) -- (v12.center) -- cycle;
		 \draw[step=1,thin] (v17.center) -- (v14.center) -- (v7.center) -- (v5.center) -- (v13.center) -- cycle;
		 \draw[step=1,thin] (v15.center) -- (v2.center) -- (v7.center) -- (v14.center) -- cycle;
		 \draw[step=1,thin] (v10.center) -- (v19.center) -- (v22.center) -- (v18.center) -- (v16.center) -- cycle;
		 \draw[step=1,thin] (v22.center) -- (v26.center) -- (v25.center) -- (v18.center) -- cycle;
		 \draw[step=1,thin] (v23.center) -- (v21.center) -- (v14.center) -- (v17.center) -- (v20.center) -- cycle;
		 \draw[step=1,thin] (v19.center) -- (v21.center) -- (v23.center) -- (v26.center) -- (v22.center) -- cycle;
		 \draw[step=1,thin] (v25.center) -- (v20.center) -- (v17.center) -- (v13.center) -- (v24.center) -- cycle;
		 \draw[step=1,thin] (v26.center) -- (v23.center) -- (v20.center) -- (v25.center) -- cycle;

		\node[scale=0.3,draw,black,thin,fill=white,circle,minimum size=2pt,inner sep=0] at (v0) {$\vertex_{0}$};
		\node[scale=0.3,draw,black,thin,fill=white,circle,minimum size=2pt,inner sep=0] at (v1) {$\vertex_{1}$};
		\node[scale=0.3,draw,black,thin,fill=white,circle,minimum size=2pt,inner sep=0] at (v2) {$\vertex_{2}$};
		\node[scale=0.3,draw,black,thin,fill=white,circle,minimum size=2pt,inner sep=0] at (v3) {$\vertex_{3}$};
		\node[scale=0.3,draw,black,thin,fill=white,circle,minimum size=2pt,inner sep=0] at (v4) {$\vertex_{4}$};
		\node[scale=0.3,draw,black,thin,fill=white,circle,minimum size=2pt,inner sep=0] at (v5) {$\vertex_{5}$};
		\node[scale=0.3,draw,black,thin,fill=white,circle,minimum size=2pt,inner sep=0] at (v6) {$\vertex_{6}$};
		\node[scale=0.3,draw,black,thin,fill=white,circle,minimum size=2pt,inner sep=0] at (v7) {$\vertex_{7}$};
		\node[scale=0.3,draw,black,thin,fill=white,circle,minimum size=2pt,inner sep=0] at (v8) {$\vertex_{8}$};
		\node[scale=0.3,draw,black,thin,fill=white,circle,minimum size=2pt,inner sep=0] at (v9) {$\vertex_{9}$};
		\node[scale=0.3,draw,black,thin,fill=white,circle,minimum size=2pt,inner sep=0] at (v10) {$\vertex_{10}$};
		\node[scale=0.3,draw,black,thin,fill=white,circle,minimum size=2pt,inner sep=0] at (v11) {$\vertex_{11}$};
		\node[scale=0.3,draw,black,thin,fill=white,circle,minimum size=2pt,inner sep=0] at (v12) {$\vertex_{12}$};
		\node[scale=0.3,draw,black,thin,fill=white,circle,minimum size=2pt,inner sep=0] at (v13) {$\vertex_{13}$};
		\node[scale=0.3,draw,black,thin,fill=white,circle,minimum size=2pt,inner sep=0] at (v14) {$\vertex_{14}$};
		\node[scale=0.3,draw,black,thin,fill=white,circle,minimum size=2pt,inner sep=0] at (v15) {$\vertex_{15}$};
		\node[scale=0.3,draw,black,thin,fill=white,circle,minimum size=2pt,inner sep=0] at (v16) {$\vertex_{16}$};
		\node[scale=0.3,draw,black,thin,fill=white,circle,minimum size=2pt,inner sep=0] at (v17) {$\vertex_{17}$};
		\node[scale=0.3,draw,black,thin,fill=white,circle,minimum size=2pt,inner sep=0] at (v18) {$\vertex_{18}$};
		\node[scale=0.3,draw,black,thin,fill=white,circle,minimum size=2pt,inner sep=0] at (v19) {$\vertex_{19}$};
		\node[scale=0.3,draw,black,thin,fill=white,circle,minimum size=2pt,inner sep=0] at (v20) {$\vertex_{20}$};
		\node[scale=0.3,draw,black,thin,fill=white,circle,minimum size=2pt,inner sep=0] at (v21) {$\vertex_{21}$};
		\node[scale=0.3,draw,black,thin,fill=white,circle,minimum size=2pt,inner sep=0] at (v22) {$\vertex_{22}$};
		\node[scale=0.3,draw,black,thin,fill=white,circle,minimum size=2pt,inner sep=0] at (v23) {$\vertex_{23}$};
		\node[scale=0.3,draw,black,thin,fill=white,circle,minimum size=2pt,inner sep=0] at (v24) {$\vertex_{24}$};
		\node[scale=0.3,draw,black,thin,fill=white,circle,minimum size=2pt,inner sep=0] at (v25) {$\vertex_{25}$};
		\node[scale=0.3,draw,black,thin,fill=white,circle,minimum size=2pt,inner sep=0] at (v26) {$\vertex_{26}$};

		\node[scale=0.3,below,inner sep=0] at ($(v0)!0.5!(v4)$) {$\edge_{0}(-1)$};
		\node[scale=0.3,left,inner sep=0] at ($(v1)!0.5!(v11)$) {$\edge_{1}(-1)$};
		\node[scale=0.3,below,inner sep=0] at ($(v3)!0.5!(v5)$) {$\edge_{2}(-1)$};
		\node[scale=0.3,left,inner sep=0] at ($(v0)!0.5!(v8)$) {$\edge_{3}(-1)$};
		\node[scale=0.3,below,inner sep=0] at ($(v4)!0.5!(v6)$) {$\edge_{4}(-1)$};
		\node[scale=0.3,below,inner sep=0] at ($(v5)!0.5!(v7)$) {$\edge_{5}(-1)$};
		\node[scale=0.3,left,inner sep=0] at ($(v4)!0.5!(v9)$) {$\edge_{6}(2)$};
		\node[scale=0.3,below,inner sep=0] at ($(v6)!0.5!(v1)$) {$\edge_{7}(-1)$};
		\node[scale=0.3,below,inner sep=0] at ($(v7)!0.5!(v2)$) {$\edge_{8}(-1)$};
		\node[scale=0.3,left,inner sep=0] at ($(v6)!0.5!(v10)$) {$\edge_{9}(1)$};
		\node[scale=0.3,left,inner sep=0] at ($(v8)!0.5!(v12)$) {$\edge_{10}(-1)$};
		\node[scale=0.3,left,inner sep=0] at ($(v9)!0.5!(v24)$) {$\edge_{11}(2)$};
		\node[scale=0.3,below,inner sep=0] at ($(v8)!0.5!(v9)$) {$\edge_{12}(0)$};
		\node[scale=0.3,left,inner sep=0] at ($(v10)!0.5!(v19)$) {$\edge_{13}(1)$};
		\node[scale=0.3,below,inner sep=0] at ($(v9)!0.5!(v16)$) {$\edge_{14}(0)$};
		\node[scale=0.3,left,inner sep=0] at ($(v11)!0.5!(v15)$) {$\edge_{15}(-1)$};
		\node[scale=0.3,below,inner sep=0] at ($(v10)!0.5!(v11)$) {$\edge_{16}(0)$};
		\node[scale=0.3,left,inner sep=0] at ($(v12)!0.5!(v3)$) {$\edge_{17}(-1)$};
		\node[scale=0.3,left,inner sep=0] at ($(v13)!0.5!(v5)$) {$\edge_{18}(2)$};
		\node[scale=0.3,below,inner sep=0] at ($(v12)!0.5!(v13)$) {$\edge_{19}(2)$};
		\node[scale=0.3,left,inner sep=0] at ($(v14)!0.5!(v7)$) {$\edge_{20}(1)$};
		\node[scale=0.3,below,inner sep=0] at ($(v13)!0.5!(v17)$) {$\edge_{21}(2)$};
		\node[scale=0.3,left,inner sep=0] at ($(v15)!0.5!(v2)$) {$\edge_{22}(-1)$};
		\node[scale=0.3,below,inner sep=0] at ($(v14)!0.5!(v15)$) {$\edge_{23}(2)$};
		\node[scale=0.3,below,inner sep=0] at ($(v16)!0.5!(v10)$) {$\edge_{24}(0)$};
		\node[scale=0.3,below,inner sep=0] at ($(v17)!0.5!(v14)$) {$\edge_{25}(2)$};
		\node[scale=0.3,left,inner sep=0] at ($(v16)!0.5!(v18)$) {$\edge_{26}(1)$};
		\node[scale=0.3,left,inner sep=0] at ($(v18)!0.5!(v25)$) {$\color{blue}\edge_{27}(0)$};
		\node[scale=0.3,left,inner sep=0] at ($(v19)!0.5!(v21)$) {$\edge_{28}(1)$};
		\node[scale=0.3,below,inner sep=0] at ($(v18)!0.5!(v22)$) {$\edge_{29}(0)$};
		\node[scale=0.3,left,inner sep=0] at ($(v20)!0.5!(v17)$) {$\edge_{30}(1)$};
		\node[scale=0.3,left,inner sep=0] at ($(v21)!0.5!(v14)$) {$\edge_{31}(1)$};
		\node[scale=0.3,below,inner sep=0] at ($(v20)!0.5!(v23)$) {$\edge_{32}(2)$};
		\node[scale=0.3,below,inner sep=0] at ($(v22)!0.5!(v19)$) {$\edge_{33}(0)$};
		\node[scale=0.3,below,inner sep=0] at ($(v23)!0.5!(v21)$) {$\edge_{34}(2)$};
		\node[scale=0.3,left,inner sep=0] at ($(v22)!0.5!(v26)$) {$\edge_{35}(2)$};
		\node[scale=0.3,left,inner sep=0] at ($(v24)!0.5!(v13)$) {$\edge_{36}(2)$};
		\node[scale=0.3,left,inner sep=0] at ($(v25)!0.5!(v20)$) {$\color{blue}\edge_{37}(0)$};
		\node[scale=0.3,below,inner sep=0] at ($(v24)!0.5!(v25)$) {$\edge_{38}(2)$};
		\node[scale=0.3,left,inner sep=0] at ($(v26)!0.5!(v23)$) {$\edge_{39}(2)$};
		\node[scale=0.3,below,inner sep=0] at ($(v25)!0.5!(v26)$) {$\edge_{40}(2)$};

	 \end{scope}
\end{tikzpicture}

%
}\\
	\subcaptionbox{$\splSpace^{\bsmoothr_1}$}[0.49\textwidth]{%
	\tikzsetnextfilename{./tikz/images/mixed_reg_ex1c}%
%
%

\begin{tikzpicture}[scale=1.7, transform shape]
	 \begin{scope}
		\coordinate (v0) at (0.0000000000000000,0.0000000000000000) {};
		\coordinate (v1) at (4.0000000000000000,0.0000000000000000) {};
		\coordinate (v2) at (4.0000000000000000,4.0000000000000000) {};
		\coordinate (v3) at (0.0000000000000000,4.0000000000000000) {};
		\coordinate (v4) at (1.0000000000000000,0.0000000000000000) {};
		\coordinate (v5) at (1.0000000000000000,4.0000000000000000) {};
		\coordinate (v6) at (3.0000000000000000,0.0000000000000000) {};
		\coordinate (v7) at (3.0000000000000000,4.0000000000000000) {};
		\coordinate (v8) at (0.0000000000000000,1.0000000000000000) {};
		\coordinate (v9) at (1.0000000000000000,1.0000000000000000) {};
		\coordinate (v10) at (3.0000000000000000,1.0000000000000000) {};
		\coordinate (v11) at (4.0000000000000000,1.0000000000000000) {};
		\coordinate (v12) at (0.0000000000000000,3.0000000000000000) {};
		\coordinate (v13) at (1.0000000000000000,3.0000000000000000) {};
		\coordinate (v14) at (3.0000000000000000,3.0000000000000000) {};
		\coordinate (v15) at (4.0000000000000000,3.0000000000000000) {};
		\coordinate (v16) at (2.0000000000000000,1.0000000000000000) {};
		\coordinate (v17) at (2.0000000000000000,3.0000000000000000) {};
		\coordinate (v18) at (2.0000000000000000,1.5000000000000000) {};
		\coordinate (v19) at (3.0000000000000000,1.5000000000000000) {};
		\coordinate (v20) at (2.0000000000000000,2.5000000000000000) {};
		\coordinate (v21) at (3.0000000000000000,2.5000000000000000) {};
		\coordinate (v22) at (2.5000000000000000,1.5000000000000000) {};
		\coordinate (v23) at (2.5000000000000000,2.5000000000000000) {};
		\coordinate (v24) at (1.0000000000000000,2.0000000000000000) {};
		\coordinate (v25) at (2.0000000000000000,2.0000000000000000) {};
		\coordinate (v26) at (2.5000000000000000,2.0000000000000000) {};

		 \draw[step=1,thin] (v4.center) -- (v9.center) -- (v8.center) -- (v0.center) -- cycle;
		 \draw[step=1,thin] (v6.center) -- (v10.center) -- (v16.center) -- (v9.center) -- (v4.center) -- cycle;
		 \draw[step=1,thin] (v1.center) -- (v11.center) -- (v10.center) -- (v6.center) -- cycle;
		 \draw[step=1,thin] (v9.center) -- (v24.center) -- (v13.center) -- (v12.center) -- (v8.center) -- cycle;
		 \draw[step=1,thin] (v16.center) -- (v18.center) -- (v25.center) -- (v24.center) -- (v9.center) -- cycle;
		 \draw[step=1,thin] (v11.center) -- (v15.center) -- (v14.center) -- (v21.center) -- (v19.center) -- (v10.center) -- cycle;
		 \draw[step=1,thin] (v13.center) -- (v5.center) -- (v3.center) -- (v12.center) -- cycle;
		 \draw[step=1,thin] (v17.center) -- (v14.center) -- (v7.center) -- (v5.center) -- (v13.center) -- cycle;
		 \draw[step=1,thin] (v15.center) -- (v2.center) -- (v7.center) -- (v14.center) -- cycle;
		 \draw[step=1,thin] (v10.center) -- (v19.center) -- (v22.center) -- (v18.center) -- (v16.center) -- cycle;
		 \draw[step=1,thin] (v22.center) -- (v26.center) -- (v25.center) -- (v18.center) -- cycle;
		 \draw[step=1,thin] (v23.center) -- (v21.center) -- (v14.center) -- (v17.center) -- (v20.center) -- cycle;
		 \draw[step=1,thin] (v19.center) -- (v21.center) -- (v23.center) -- (v26.center) -- (v22.center) -- cycle;
		 \draw[step=1,thin] (v25.center) -- (v20.center) -- (v17.center) -- (v13.center) -- (v24.center) -- cycle;
		 \draw[step=1,thin] (v26.center) -- (v23.center) -- (v20.center) -- (v25.center) -- cycle;

		\node[scale=0.3,draw,black,thin,fill=white,circle,minimum size=2pt,inner sep=0] at (v0) {$\vertex_{0}$};
		\node[scale=0.3,draw,black,thin,fill=white,circle,minimum size=2pt,inner sep=0] at (v1) {$\vertex_{1}$};
		\node[scale=0.3,draw,black,thin,fill=white,circle,minimum size=2pt,inner sep=0] at (v2) {$\vertex_{2}$};
		\node[scale=0.3,draw,black,thin,fill=white,circle,minimum size=2pt,inner sep=0] at (v3) {$\vertex_{3}$};
		\node[scale=0.3,draw,black,thin,fill=white,circle,minimum size=2pt,inner sep=0] at (v4) {$\vertex_{4}$};
		\node[scale=0.3,draw,black,thin,fill=white,circle,minimum size=2pt,inner sep=0] at (v5) {$\vertex_{5}$};
		\node[scale=0.3,draw,black,thin,fill=white,circle,minimum size=2pt,inner sep=0] at (v6) {$\vertex_{6}$};
		\node[scale=0.3,draw,black,thin,fill=white,circle,minimum size=2pt,inner sep=0] at (v7) {$\vertex_{7}$};
		\node[scale=0.3,draw,black,thin,fill=white,circle,minimum size=2pt,inner sep=0] at (v8) {$\vertex_{8}$};
		\node[scale=0.3,draw,black,thin,fill=white,circle,minimum size=2pt,inner sep=0] at (v9) {$\vertex_{9}$};
		\node[scale=0.3,draw,black,thin,fill=white,circle,minimum size=2pt,inner sep=0] at (v10) {$\vertex_{10}$};
		\node[scale=0.3,draw,black,thin,fill=white,circle,minimum size=2pt,inner sep=0] at (v11) {$\vertex_{11}$};
		\node[scale=0.3,draw,black,thin,fill=white,circle,minimum size=2pt,inner sep=0] at (v12) {$\vertex_{12}$};
		\node[scale=0.3,draw,black,thin,fill=white,circle,minimum size=2pt,inner sep=0] at (v13) {$\vertex_{13}$};
		\node[scale=0.3,draw,black,thin,fill=white,circle,minimum size=2pt,inner sep=0] at (v14) {$\vertex_{14}$};
		\node[scale=0.3,draw,black,thin,fill=white,circle,minimum size=2pt,inner sep=0] at (v15) {$\vertex_{15}$};
		\node[scale=0.3,draw,black,thin,fill=white,circle,minimum size=2pt,inner sep=0] at (v16) {$\vertex_{16}$};
		\node[scale=0.3,draw,black,thin,fill=white,circle,minimum size=2pt,inner sep=0] at (v17) {$\vertex_{17}$};
		\node[scale=0.3,draw,black,thin,fill=white,circle,minimum size=2pt,inner sep=0] at (v18) {$\vertex_{18}$};
		\node[scale=0.3,draw,black,thin,fill=white,circle,minimum size=2pt,inner sep=0] at (v19) {$\vertex_{19}$};
		\node[scale=0.3,draw,black,thin,fill=white,circle,minimum size=2pt,inner sep=0] at (v20) {$\vertex_{20}$};
		\node[scale=0.3,draw,black,thin,fill=white,circle,minimum size=2pt,inner sep=0] at (v21) {$\vertex_{21}$};
		\node[scale=0.3,draw,black,thin,fill=white,circle,minimum size=2pt,inner sep=0] at (v22) {$\vertex_{22}$};
		\node[scale=0.3,draw,black,thin,fill=white,circle,minimum size=2pt,inner sep=0] at (v23) {$\vertex_{23}$};
		\node[scale=0.3,draw,black,thin,fill=white,circle,minimum size=2pt,inner sep=0] at (v24) {$\vertex_{24}$};
		\node[scale=0.3,draw,black,thin,fill=white,circle,minimum size=2pt,inner sep=0] at (v25) {$\vertex_{25}$};
		\node[scale=0.3,draw,black,thin,fill=white,circle,minimum size=2pt,inner sep=0] at (v26) {$\vertex_{26}$};

		\node[scale=0.3,below,inner sep=0] at ($(v0)!0.5!(v4)$) {$\edge_{0}(-1)$};
		\node[scale=0.3,left,inner sep=0] at ($(v1)!0.5!(v11)$) {$\edge_{1}(-1)$};
		\node[scale=0.3,below,inner sep=0] at ($(v3)!0.5!(v5)$) {$\edge_{2}(-1)$};
		\node[scale=0.3,left,inner sep=0] at ($(v0)!0.5!(v8)$) {$\edge_{3}(-1)$};
		\node[scale=0.3,below,inner sep=0] at ($(v4)!0.5!(v6)$) {$\edge_{4}(-1)$};
		\node[scale=0.3,below,inner sep=0] at ($(v5)!0.5!(v7)$) {$\edge_{5}(-1)$};
		\node[scale=0.3,left,inner sep=0] at ($(v4)!0.5!(v9)$) {$\edge_{6}(2)$};
		\node[scale=0.3,below,inner sep=0] at ($(v6)!0.5!(v1)$) {$\edge_{7}(-1)$};
		\node[scale=0.3,below,inner sep=0] at ($(v7)!0.5!(v2)$) {$\edge_{8}(-1)$};
		\node[scale=0.3,left,inner sep=0] at ($(v6)!0.5!(v10)$) {$\edge_{9}(1)$};
		\node[scale=0.3,left,inner sep=0] at ($(v8)!0.5!(v12)$) {$\edge_{10}(-1)$};
		\node[scale=0.3,left,inner sep=0] at ($(v9)!0.5!(v24)$) {$\edge_{11}(2)$};
		\node[scale=0.3,below,inner sep=0] at ($(v8)!0.5!(v9)$) {$\edge_{12}(0)$};
		\node[scale=0.3,left,inner sep=0] at ($(v10)!0.5!(v19)$) {$\edge_{13}(1)$};
		\node[scale=0.3,below,inner sep=0] at ($(v9)!0.5!(v16)$) {$\edge_{14}(0)$};
		\node[scale=0.3,left,inner sep=0] at ($(v11)!0.5!(v15)$) {$\edge_{15}(-1)$};
		\node[scale=0.3,below,inner sep=0] at ($(v10)!0.5!(v11)$) {$\edge_{16}(0)$};
		\node[scale=0.3,left,inner sep=0] at ($(v12)!0.5!(v3)$) {$\edge_{17}(-1)$};
		\node[scale=0.3,left,inner sep=0] at ($(v13)!0.5!(v5)$) {$\edge_{18}(2)$};
		\node[scale=0.3,below,inner sep=0] at ($(v12)!0.5!(v13)$) {$\edge_{19}(2)$};
		\node[scale=0.3,left,inner sep=0] at ($(v14)!0.5!(v7)$) {$\edge_{20}(1)$};
		\node[scale=0.3,below,inner sep=0] at ($(v13)!0.5!(v17)$) {$\edge_{21}(2)$};
		\node[scale=0.3,left,inner sep=0] at ($(v15)!0.5!(v2)$) {$\edge_{22}(-1)$};
		\node[scale=0.3,below,inner sep=0] at ($(v14)!0.5!(v15)$) {$\edge_{23}(2)$};
		\node[scale=0.3,below,inner sep=0] at ($(v16)!0.5!(v10)$) {$\edge_{24}(0)$};
		\node[scale=0.3,below,inner sep=0] at ($(v17)!0.5!(v14)$) {$\edge_{25}(2)$};
		\node[scale=0.3,left,inner sep=0] at ($(v16)!0.5!(v18)$) {$\edge_{26}(1)$};
		\node[scale=0.3,left,inner sep=0] at ($(v18)!0.5!(v25)$) {$\color{blue}\edge_{27}(0)$};
		\node[scale=0.3,left,inner sep=0] at ($(v19)!0.5!(v21)$) {$\edge_{28}(1)$};
		\node[scale=0.3,below,inner sep=0] at ($(v18)!0.5!(v22)$) {$\edge_{29}(0)$};
		\node[scale=0.3,left,inner sep=0] at ($(v20)!0.5!(v17)$) {$\color{blue}\edge_{30}(0)$};
		\node[scale=0.3,left,inner sep=0] at ($(v21)!0.5!(v14)$) {$\edge_{31}(1)$};
		\node[scale=0.3,below,inner sep=0] at ($(v20)!0.5!(v23)$) {$\edge_{32}(2)$};
		\node[scale=0.3,below,inner sep=0] at ($(v22)!0.5!(v19)$) {$\edge_{33}(0)$};
		\node[scale=0.3,below,inner sep=0] at ($(v23)!0.5!(v21)$) {$\edge_{34}(2)$};
		\node[scale=0.3,left,inner sep=0] at ($(v22)!0.5!(v26)$) {$\edge_{35}(2)$};
		\node[scale=0.3,left,inner sep=0] at ($(v24)!0.5!(v13)$) {$\edge_{36}(2)$};
		\node[scale=0.3,left,inner sep=0] at ($(v25)!0.5!(v20)$) {$\color{blue}\edge_{37}(0)$};
		\node[scale=0.3,below,inner sep=0] at ($(v24)!0.5!(v25)$) {$\edge_{38}(2)$};
		\node[scale=0.3,left,inner sep=0] at ($(v26)!0.5!(v23)$) {$\edge_{39}(2)$};
		\node[scale=0.3,below,inner sep=0] at ($(v25)!0.5!(v26)$) {$\edge_{40}(2)$};

	 \end{scope}
\end{tikzpicture}

%
}
	\subcaptionbox{$\splSpace^{\bsmoothr_2}$}[0.49\textwidth]{%
	\tikzsetnextfilename{./tikz/images/mixed_reg_ex1d}%
%
%
	
	\begin{tikzpicture}[scale=1.7, transform shape]
	\begin{scope}
	\coordinate (v0) at (0.0000000000000000,0.0000000000000000) {};
	\coordinate (v1) at (4.0000000000000000,0.0000000000000000) {};
	\coordinate (v2) at (4.0000000000000000,4.0000000000000000) {};
	\coordinate (v3) at (0.0000000000000000,4.0000000000000000) {};
	\coordinate (v4) at (1.0000000000000000,0.0000000000000000) {};
	\coordinate (v5) at (1.0000000000000000,4.0000000000000000) {};
	\coordinate (v6) at (3.0000000000000000,0.0000000000000000) {};
	\coordinate (v7) at (3.0000000000000000,4.0000000000000000) {};
	\coordinate (v8) at (0.0000000000000000,1.0000000000000000) {};
	\coordinate (v9) at (1.0000000000000000,1.0000000000000000) {};
	\coordinate (v10) at (3.0000000000000000,1.0000000000000000) {};
	\coordinate (v11) at (4.0000000000000000,1.0000000000000000) {};
	\coordinate (v12) at (0.0000000000000000,3.0000000000000000) {};
	\coordinate (v13) at (1.0000000000000000,3.0000000000000000) {};
	\coordinate (v14) at (3.0000000000000000,3.0000000000000000) {};
	\coordinate (v15) at (4.0000000000000000,3.0000000000000000) {};
	\coordinate (v16) at (2.0000000000000000,1.0000000000000000) {};
	\coordinate (v17) at (2.0000000000000000,3.0000000000000000) {};
	\coordinate (v18) at (2.0000000000000000,1.5000000000000000) {};
	\coordinate (v19) at (3.0000000000000000,1.5000000000000000) {};
	\coordinate (v20) at (2.0000000000000000,2.5000000000000000) {};
	\coordinate (v21) at (3.0000000000000000,2.5000000000000000) {};
	\coordinate (v22) at (2.5000000000000000,1.5000000000000000) {};
	\coordinate (v23) at (2.5000000000000000,2.5000000000000000) {};
	\coordinate (v24) at (1.0000000000000000,2.0000000000000000) {};
	\coordinate (v25) at (2.0000000000000000,2.0000000000000000) {};
	\coordinate (v26) at (2.5000000000000000,2.0000000000000000) {};
	
	\draw[step=1,thin] (v4.center) -- (v9.center) -- (v8.center) -- (v0.center) -- cycle;
	\draw[step=1,thin] (v6.center) -- (v10.center) -- (v16.center) -- (v9.center) -- (v4.center) -- cycle;
	\draw[step=1,thin] (v1.center) -- (v11.center) -- (v10.center) -- (v6.center) -- cycle;
	\draw[step=1,thin] (v9.center) -- (v24.center) -- (v13.center) -- (v12.center) -- (v8.center) -- cycle;
	\draw[step=1,thin] (v16.center) -- (v18.center) -- (v25.center) -- (v24.center) -- (v9.center) -- cycle;
	\draw[step=1,thin] (v11.center) -- (v15.center) -- (v14.center) -- (v21.center) -- (v19.center) -- (v10.center) -- cycle;
	\draw[step=1,thin] (v13.center) -- (v5.center) -- (v3.center) -- (v12.center) -- cycle;
	\draw[step=1,thin] (v17.center) -- (v14.center) -- (v7.center) -- (v5.center) -- (v13.center) -- cycle;
	\draw[step=1,thin] (v15.center) -- (v2.center) -- (v7.center) -- (v14.center) -- cycle;
	\draw[step=1,thin] (v10.center) -- (v19.center) -- (v22.center) -- (v18.center) -- (v16.center) -- cycle;
	\draw[step=1,thin] (v22.center) -- (v26.center) -- (v25.center) -- (v18.center) -- cycle;
	\draw[step=1,thin] (v23.center) -- (v21.center) -- (v14.center) -- (v17.center) -- (v20.center) -- cycle;
	\draw[step=1,thin] (v19.center) -- (v21.center) -- (v23.center) -- (v26.center) -- (v22.center) -- cycle;
	\draw[step=1,thin] (v25.center) -- (v20.center) -- (v17.center) -- (v13.center) -- (v24.center) -- cycle;
	\draw[step=1,thin] (v26.center) -- (v23.center) -- (v20.center) -- (v25.center) -- cycle;
	
	\node[scale=0.3,draw,black,thin,fill=white,circle,minimum size=2pt,inner sep=0] at (v0) {$\vertex_{0}$};
	\node[scale=0.3,draw,black,thin,fill=white,circle,minimum size=2pt,inner sep=0] at (v1) {$\vertex_{1}$};
	\node[scale=0.3,draw,black,thin,fill=white,circle,minimum size=2pt,inner sep=0] at (v2) {$\vertex_{2}$};
	\node[scale=0.3,draw,black,thin,fill=white,circle,minimum size=2pt,inner sep=0] at (v3) {$\vertex_{3}$};
	\node[scale=0.3,draw,black,thin,fill=white,circle,minimum size=2pt,inner sep=0] at (v4) {$\vertex_{4}$};
	\node[scale=0.3,draw,black,thin,fill=white,circle,minimum size=2pt,inner sep=0] at (v5) {$\vertex_{5}$};
	\node[scale=0.3,draw,black,thin,fill=white,circle,minimum size=2pt,inner sep=0] at (v6) {$\vertex_{6}$};
	\node[scale=0.3,draw,black,thin,fill=white,circle,minimum size=2pt,inner sep=0] at (v7) {$\vertex_{7}$};
	\node[scale=0.3,draw,black,thin,fill=white,circle,minimum size=2pt,inner sep=0] at (v8) {$\vertex_{8}$};
	\node[scale=0.3,draw,black,thin,fill=white,circle,minimum size=2pt,inner sep=0] at (v9) {$\vertex_{9}$};
	\node[scale=0.3,draw,black,thin,fill=white,circle,minimum size=2pt,inner sep=0] at (v10) {$\vertex_{10}$};
	\node[scale=0.3,draw,black,thin,fill=white,circle,minimum size=2pt,inner sep=0] at (v11) {$\vertex_{11}$};
	\node[scale=0.3,draw,black,thin,fill=white,circle,minimum size=2pt,inner sep=0] at (v12) {$\vertex_{12}$};
	\node[scale=0.3,draw,black,thin,fill=white,circle,minimum size=2pt,inner sep=0] at (v13) {$\vertex_{13}$};
	\node[scale=0.3,draw,black,thin,fill=white,circle,minimum size=2pt,inner sep=0] at (v14) {$\vertex_{14}$};
	\node[scale=0.3,draw,black,thin,fill=white,circle,minimum size=2pt,inner sep=0] at (v15) {$\vertex_{15}$};
	\node[scale=0.3,draw,black,thin,fill=white,circle,minimum size=2pt,inner sep=0] at (v16) {$\vertex_{16}$};
	\node[scale=0.3,draw,black,thin,fill=white,circle,minimum size=2pt,inner sep=0] at (v17) {$\vertex_{17}$};
	\node[scale=0.3,draw,black,thin,fill=white,circle,minimum size=2pt,inner sep=0] at (v18) {$\vertex_{18}$};
	\node[scale=0.3,draw,black,thin,fill=white,circle,minimum size=2pt,inner sep=0] at (v19) {$\vertex_{19}$};
	\node[scale=0.3,draw,black,thin,fill=white,circle,minimum size=2pt,inner sep=0] at (v20) {$\vertex_{20}$};
	\node[scale=0.3,draw,black,thin,fill=white,circle,minimum size=2pt,inner sep=0] at (v21) {$\vertex_{21}$};
	\node[scale=0.3,draw,black,thin,fill=white,circle,minimum size=2pt,inner sep=0] at (v22) {$\vertex_{22}$};
	\node[scale=0.3,draw,black,thin,fill=white,circle,minimum size=2pt,inner sep=0] at (v23) {$\vertex_{23}$};
	\node[scale=0.3,draw,black,thin,fill=white,circle,minimum size=2pt,inner sep=0] at (v24) {$\vertex_{24}$};
	\node[scale=0.3,draw,black,thin,fill=white,circle,minimum size=2pt,inner sep=0] at (v25) {$\vertex_{25}$};
	\node[scale=0.3,draw,black,thin,fill=white,circle,minimum size=2pt,inner sep=0] at (v26) {$\vertex_{26}$};
	
	\node[scale=0.3,below,inner sep=0] at ($(v0)!0.5!(v4)$) {$\edge_{0}(-1)$};
	\node[scale=0.3,left,inner sep=0] at ($(v1)!0.5!(v11)$) {$\edge_{1}(-1)$};
	\node[scale=0.3,below,inner sep=0] at ($(v3)!0.5!(v5)$) {$\edge_{2}(-1)$};
	\node[scale=0.3,left,inner sep=0] at ($(v0)!0.5!(v8)$) {$\edge_{3}(-1)$};
	\node[scale=0.3,below,inner sep=0] at ($(v4)!0.5!(v6)$) {$\edge_{4}(-1)$};
	\node[scale=0.3,below,inner sep=0] at ($(v5)!0.5!(v7)$) {$\edge_{5}(-1)$};
	\node[scale=0.3,left,inner sep=0] at ($(v4)!0.5!(v9)$) {$\edge_{6}(2)$};
	\node[scale=0.3,below,inner sep=0] at ($(v6)!0.5!(v1)$) {$\edge_{7}(-1)$};
	\node[scale=0.3,below,inner sep=0] at ($(v7)!0.5!(v2)$) {$\edge_{8}(-1)$};
	\node[scale=0.3,left,inner sep=0] at ($(v6)!0.5!(v10)$) {$\edge_{9}(1)$};
	\node[scale=0.3,left,inner sep=0] at ($(v8)!0.5!(v12)$) {$\edge_{10}(-1)$};
	\node[scale=0.3,left,inner sep=0] at ($(v9)!0.5!(v24)$) {$\edge_{11}(2)$};
	\node[scale=0.3,below,inner sep=0] at ($(v8)!0.5!(v9)$) {$\edge_{12}(0)$};
	\node[scale=0.3,left,inner sep=0] at ($(v10)!0.5!(v19)$) {$\edge_{13}(1)$};
	\node[scale=0.3,below,inner sep=0] at ($(v9)!0.5!(v16)$) {$\edge_{14}(0)$};
	\node[scale=0.3,left,inner sep=0] at ($(v11)!0.5!(v15)$) {$\edge_{15}(-1)$};
	\node[scale=0.3,below,inner sep=0] at ($(v10)!0.5!(v11)$) {$\edge_{16}(0)$};
	\node[scale=0.3,left,inner sep=0] at ($(v12)!0.5!(v3)$) {$\edge_{17}(-1)$};
	\node[scale=0.3,left,inner sep=0] at ($(v13)!0.5!(v5)$) {$\edge_{18}(2)$};
	\node[scale=0.3,below,inner sep=0] at ($(v12)!0.5!(v13)$) {$\edge_{19}(2)$};
	\node[scale=0.3,left,inner sep=0] at ($(v14)!0.5!(v7)$) {$\edge_{20}(1)$};
	\node[scale=0.3,below,inner sep=0] at ($(v13)!0.5!(v17)$) {$\edge_{21}(2)$};
	\node[scale=0.3,left,inner sep=0] at ($(v15)!0.5!(v2)$) {$\edge_{22}(-1)$};
	\node[scale=0.3,below,inner sep=0] at ($(v14)!0.5!(v15)$) {$\edge_{23}(2)$};
	\node[scale=0.3,below,inner sep=0] at ($(v16)!0.5!(v10)$) {$\edge_{24}(0)$};
	\node[scale=0.3,below,inner sep=0] at ($(v17)!0.5!(v14)$) {$\edge_{25}(2)$};
	\node[scale=0.3,left,inner sep=0] at ($(v16)!0.5!(v18)$) {$\edge_{26}(1)$};
	\node[scale=0.3,left,inner sep=0] at ($(v18)!0.5!(v25)$) {$\color{blue}\edge_{27}(0)$};
	\node[scale=0.3,left,inner sep=0] at ($(v19)!0.5!(v21)$) {$\edge_{28}(1)$};
	\node[scale=0.3,below,inner sep=0] at ($(v18)!0.5!(v22)$) {$\edge_{29}(0)$};
	\node[scale=0.3,left,inner sep=0] at ($(v20)!0.5!(v17)$) {$\color{blue}\edge_{30}(0)$};
	\node[scale=0.3,left,inner sep=0] at ($(v21)!0.5!(v14)$) {$\edge_{31}(1)$};
	\node[scale=0.3,below,inner sep=0] at ($(v20)!0.5!(v23)$) {$\color{blue}\edge_{32}(1)$};
	\node[scale=0.3,below,inner sep=0] at ($(v22)!0.5!(v19)$) {$\edge_{33}(0)$};
	\node[scale=0.3,below,inner sep=0] at ($(v23)!0.5!(v21)$) {$\edge_{34}(2)$};
	\node[scale=0.3,left,inner sep=0] at ($(v22)!0.5!(v26)$) {$\edge_{35}(2)$};
	\node[scale=0.3,left,inner sep=0] at ($(v24)!0.5!(v13)$) {$\edge_{36}(2)$};
	\node[scale=0.3,left,inner sep=0] at ($(v25)!0.5!(v20)$) {$\color{blue}\edge_{37}(0)$};
	\node[scale=0.3,below,inner sep=0] at ($(v24)!0.5!(v25)$) {$\edge_{38}(2)$};
	\node[scale=0.3,left,inner sep=0] at ($(v26)!0.5!(v23)$) {$\edge_{39}(2)$};
	\node[scale=0.3,below,inner sep=0] at ($(v25)!0.5!(v26)$) {$\edge_{40}(2)$};
	
	\end{scope}
	\end{tikzpicture}
	
%
}
	\caption{The above figures correspond to the bi-cubic spline space considered in Example \ref{ex:ex2}.
		The smoothness required across each edge has been annotated in parenthesis next to the edge labels.
		The smoothness distributions in Figures (b)--(f) differ from the one in Figure (a) only on the edges labelled in blue.
	}
	\label{fig:ex2}
\end{figure}
\begin{figure*}\ContinuedFloat
	\subcaptionbox{$\splSpace^{\bsmoothr_3}$}[0.49\textwidth]{%
	\tikzsetnextfilename{./tikz/images/mixed_reg_ex1e}%
%
%

\begin{tikzpicture}[scale=1.7, transform shape]
	 \begin{scope}
		\coordinate (v0) at (0.0000000000000000,0.0000000000000000) {};
		\coordinate (v1) at (4.0000000000000000,0.0000000000000000) {};
		\coordinate (v2) at (4.0000000000000000,4.0000000000000000) {};
		\coordinate (v3) at (0.0000000000000000,4.0000000000000000) {};
		\coordinate (v4) at (1.0000000000000000,0.0000000000000000) {};
		\coordinate (v5) at (1.0000000000000000,4.0000000000000000) {};
		\coordinate (v6) at (3.0000000000000000,0.0000000000000000) {};
		\coordinate (v7) at (3.0000000000000000,4.0000000000000000) {};
		\coordinate (v8) at (0.0000000000000000,1.0000000000000000) {};
		\coordinate (v9) at (1.0000000000000000,1.0000000000000000) {};
		\coordinate (v10) at (3.0000000000000000,1.0000000000000000) {};
		\coordinate (v11) at (4.0000000000000000,1.0000000000000000) {};
		\coordinate (v12) at (0.0000000000000000,3.0000000000000000) {};
		\coordinate (v13) at (1.0000000000000000,3.0000000000000000) {};
		\coordinate (v14) at (3.0000000000000000,3.0000000000000000) {};
		\coordinate (v15) at (4.0000000000000000,3.0000000000000000) {};
		\coordinate (v16) at (2.0000000000000000,1.0000000000000000) {};
		\coordinate (v17) at (2.0000000000000000,3.0000000000000000) {};
		\coordinate (v18) at (2.0000000000000000,1.5000000000000000) {};
		\coordinate (v19) at (3.0000000000000000,1.5000000000000000) {};
		\coordinate (v20) at (2.0000000000000000,2.5000000000000000) {};
		\coordinate (v21) at (3.0000000000000000,2.5000000000000000) {};
		\coordinate (v22) at (2.5000000000000000,1.5000000000000000) {};
		\coordinate (v23) at (2.5000000000000000,2.5000000000000000) {};
		\coordinate (v24) at (1.0000000000000000,2.0000000000000000) {};
		\coordinate (v25) at (2.0000000000000000,2.0000000000000000) {};
		\coordinate (v26) at (2.5000000000000000,2.0000000000000000) {};

		 \draw[step=1,thin] (v4.center) -- (v9.center) -- (v8.center) -- (v0.center) -- cycle;
		 \draw[step=1,thin] (v6.center) -- (v10.center) -- (v16.center) -- (v9.center) -- (v4.center) -- cycle;
		 \draw[step=1,thin] (v1.center) -- (v11.center) -- (v10.center) -- (v6.center) -- cycle;
		 \draw[step=1,thin] (v9.center) -- (v24.center) -- (v13.center) -- (v12.center) -- (v8.center) -- cycle;
		 \draw[step=1,thin] (v16.center) -- (v18.center) -- (v25.center) -- (v24.center) -- (v9.center) -- cycle;
		 \draw[step=1,thin] (v11.center) -- (v15.center) -- (v14.center) -- (v21.center) -- (v19.center) -- (v10.center) -- cycle;
		 \draw[step=1,thin] (v13.center) -- (v5.center) -- (v3.center) -- (v12.center) -- cycle;
		 \draw[step=1,thin] (v17.center) -- (v14.center) -- (v7.center) -- (v5.center) -- (v13.center) -- cycle;
		 \draw[step=1,thin] (v15.center) -- (v2.center) -- (v7.center) -- (v14.center) -- cycle;
		 \draw[step=1,thin] (v10.center) -- (v19.center) -- (v22.center) -- (v18.center) -- (v16.center) -- cycle;
		 \draw[step=1,thin] (v22.center) -- (v26.center) -- (v25.center) -- (v18.center) -- cycle;
		 \draw[step=1,thin] (v23.center) -- (v21.center) -- (v14.center) -- (v17.center) -- (v20.center) -- cycle;
		 \draw[step=1,thin] (v19.center) -- (v21.center) -- (v23.center) -- (v26.center) -- (v22.center) -- cycle;
		 \draw[step=1,thin] (v25.center) -- (v20.center) -- (v17.center) -- (v13.center) -- (v24.center) -- cycle;
		 \draw[step=1,thin] (v26.center) -- (v23.center) -- (v20.center) -- (v25.center) -- cycle;

		\node[scale=0.3,draw,black,thin,fill=white,circle,minimum size=2pt,inner sep=0] at (v0) {$\vertex_{0}$};
		\node[scale=0.3,draw,black,thin,fill=white,circle,minimum size=2pt,inner sep=0] at (v1) {$\vertex_{1}$};
		\node[scale=0.3,draw,black,thin,fill=white,circle,minimum size=2pt,inner sep=0] at (v2) {$\vertex_{2}$};
		\node[scale=0.3,draw,black,thin,fill=white,circle,minimum size=2pt,inner sep=0] at (v3) {$\vertex_{3}$};
		\node[scale=0.3,draw,black,thin,fill=white,circle,minimum size=2pt,inner sep=0] at (v4) {$\vertex_{4}$};
		\node[scale=0.3,draw,black,thin,fill=white,circle,minimum size=2pt,inner sep=0] at (v5) {$\vertex_{5}$};
		\node[scale=0.3,draw,black,thin,fill=white,circle,minimum size=2pt,inner sep=0] at (v6) {$\vertex_{6}$};
		\node[scale=0.3,draw,black,thin,fill=white,circle,minimum size=2pt,inner sep=0] at (v7) {$\vertex_{7}$};
		\node[scale=0.3,draw,black,thin,fill=white,circle,minimum size=2pt,inner sep=0] at (v8) {$\vertex_{8}$};
		\node[scale=0.3,draw,black,thin,fill=white,circle,minimum size=2pt,inner sep=0] at (v9) {$\vertex_{9}$};
		\node[scale=0.3,draw,black,thin,fill=white,circle,minimum size=2pt,inner sep=0] at (v10) {$\vertex_{10}$};
		\node[scale=0.3,draw,black,thin,fill=white,circle,minimum size=2pt,inner sep=0] at (v11) {$\vertex_{11}$};
		\node[scale=0.3,draw,black,thin,fill=white,circle,minimum size=2pt,inner sep=0] at (v12) {$\vertex_{12}$};
		\node[scale=0.3,draw,black,thin,fill=white,circle,minimum size=2pt,inner sep=0] at (v13) {$\vertex_{13}$};
		\node[scale=0.3,draw,black,thin,fill=white,circle,minimum size=2pt,inner sep=0] at (v14) {$\vertex_{14}$};
		\node[scale=0.3,draw,black,thin,fill=white,circle,minimum size=2pt,inner sep=0] at (v15) {$\vertex_{15}$};
		\node[scale=0.3,draw,black,thin,fill=white,circle,minimum size=2pt,inner sep=0] at (v16) {$\vertex_{16}$};
		\node[scale=0.3,draw,black,thin,fill=white,circle,minimum size=2pt,inner sep=0] at (v17) {$\vertex_{17}$};
		\node[scale=0.3,draw,black,thin,fill=white,circle,minimum size=2pt,inner sep=0] at (v18) {$\vertex_{18}$};
		\node[scale=0.3,draw,black,thin,fill=white,circle,minimum size=2pt,inner sep=0] at (v19) {$\vertex_{19}$};
		\node[scale=0.3,draw,black,thin,fill=white,circle,minimum size=2pt,inner sep=0] at (v20) {$\vertex_{20}$};
		\node[scale=0.3,draw,black,thin,fill=white,circle,minimum size=2pt,inner sep=0] at (v21) {$\vertex_{21}$};
		\node[scale=0.3,draw,black,thin,fill=white,circle,minimum size=2pt,inner sep=0] at (v22) {$\vertex_{22}$};
		\node[scale=0.3,draw,black,thin,fill=white,circle,minimum size=2pt,inner sep=0] at (v23) {$\vertex_{23}$};
		\node[scale=0.3,draw,black,thin,fill=white,circle,minimum size=2pt,inner sep=0] at (v24) {$\vertex_{24}$};
		\node[scale=0.3,draw,black,thin,fill=white,circle,minimum size=2pt,inner sep=0] at (v25) {$\vertex_{25}$};
		\node[scale=0.3,draw,black,thin,fill=white,circle,minimum size=2pt,inner sep=0] at (v26) {$\vertex_{26}$};

		\node[scale=0.3,below,inner sep=0] at ($(v0)!0.5!(v4)$) {$\edge_{0}(-1)$};
		\node[scale=0.3,left,inner sep=0] at ($(v1)!0.5!(v11)$) {$\edge_{1}(-1)$};
		\node[scale=0.3,below,inner sep=0] at ($(v3)!0.5!(v5)$) {$\edge_{2}(-1)$};
		\node[scale=0.3,left,inner sep=0] at ($(v0)!0.5!(v8)$) {$\edge_{3}(-1)$};
		\node[scale=0.3,below,inner sep=0] at ($(v4)!0.5!(v6)$) {$\edge_{4}(-1)$};
		\node[scale=0.3,below,inner sep=0] at ($(v5)!0.5!(v7)$) {$\edge_{5}(-1)$};
		\node[scale=0.3,left,inner sep=0] at ($(v4)!0.5!(v9)$) {$\edge_{6}(2)$};
		\node[scale=0.3,below,inner sep=0] at ($(v6)!0.5!(v1)$) {$\edge_{7}(-1)$};
		\node[scale=0.3,below,inner sep=0] at ($(v7)!0.5!(v2)$) {$\edge_{8}(-1)$};
		\node[scale=0.3,left,inner sep=0] at ($(v6)!0.5!(v10)$) {$\edge_{9}(1)$};
		\node[scale=0.3,left,inner sep=0] at ($(v8)!0.5!(v12)$) {$\edge_{10}(-1)$};
		\node[scale=0.3,left,inner sep=0] at ($(v9)!0.5!(v24)$) {$\edge_{11}(2)$};
		\node[scale=0.3,below,inner sep=0] at ($(v8)!0.5!(v9)$) {$\edge_{12}(0)$};
		\node[scale=0.3,left,inner sep=0] at ($(v10)!0.5!(v19)$) {$\edge_{13}(1)$};
		\node[scale=0.3,below,inner sep=0] at ($(v9)!0.5!(v16)$) {$\edge_{14}(0)$};
		\node[scale=0.3,left,inner sep=0] at ($(v11)!0.5!(v15)$) {$\edge_{15}(-1)$};
		\node[scale=0.3,below,inner sep=0] at ($(v10)!0.5!(v11)$) {$\edge_{16}(0)$};
		\node[scale=0.3,left,inner sep=0] at ($(v12)!0.5!(v3)$) {$\edge_{17}(-1)$};
		\node[scale=0.3,left,inner sep=0] at ($(v13)!0.5!(v5)$) {$\edge_{18}(2)$};
		\node[scale=0.3,below,inner sep=0] at ($(v12)!0.5!(v13)$) {$\edge_{19}(2)$};
		\node[scale=0.3,left,inner sep=0] at ($(v14)!0.5!(v7)$) {$\edge_{20}(1)$};
		\node[scale=0.3,below,inner sep=0] at ($(v13)!0.5!(v17)$) {$\edge_{21}(2)$};
		\node[scale=0.3,left,inner sep=0] at ($(v15)!0.5!(v2)$) {$\edge_{22}(-1)$};
		\node[scale=0.3,below,inner sep=0] at ($(v14)!0.5!(v15)$) {$\edge_{23}(2)$};
		\node[scale=0.3,below,inner sep=0] at ($(v16)!0.5!(v10)$) {$\edge_{24}(0)$};
		\node[scale=0.3,below,inner sep=0] at ($(v17)!0.5!(v14)$) {$\color{blue}\edge_{25}(1)$};
		\node[scale=0.3,left,inner sep=0] at ($(v16)!0.5!(v18)$) {$\edge_{26}(1)$};
		\node[scale=0.3,left,inner sep=0] at ($(v18)!0.5!(v25)$) {$\color{blue}\edge_{27}(0)$};
		\node[scale=0.3,left,inner sep=0] at ($(v19)!0.5!(v21)$) {$\edge_{28}(1)$};
		\node[scale=0.3,below,inner sep=0] at ($(v18)!0.5!(v22)$) {$\edge_{29}(0)$};
		\node[scale=0.3,left,inner sep=0] at ($(v20)!0.5!(v17)$) {$\color{blue}\edge_{30}(0)$};
		\node[scale=0.3,left,inner sep=0] at ($(v21)!0.5!(v14)$) {$\edge_{31}(1)$};
		\node[scale=0.3,below,inner sep=0] at ($(v20)!0.5!(v23)$) {$\color{blue}\edge_{32}(1)$};
		\node[scale=0.3,below,inner sep=0] at ($(v22)!0.5!(v19)$) {$\edge_{33}(0)$};
		\node[scale=0.3,below,inner sep=0] at ($(v23)!0.5!(v21)$) {$\edge_{34}(2)$};
		\node[scale=0.3,left,inner sep=0] at ($(v22)!0.5!(v26)$) {$\edge_{35}(2)$};
		\node[scale=0.3,left,inner sep=0] at ($(v24)!0.5!(v13)$) {$\edge_{36}(2)$};
		\node[scale=0.3,left,inner sep=0] at ($(v25)!0.5!(v20)$) {$\color{blue}\edge_{37}(0)$};
		\node[scale=0.3,below,inner sep=0] at ($(v24)!0.5!(v25)$) {$\edge_{38}(2)$};
		\node[scale=0.3,left,inner sep=0] at ($(v26)!0.5!(v23)$) {$\edge_{39}(2)$};
		\node[scale=0.3,below,inner sep=0] at ($(v25)!0.5!(v26)$) {$\edge_{40}(2)$};

	 \end{scope}
\end{tikzpicture}

%
}
	\subcaptionbox{$\splSpace^{\bsmoothr_4}$}[0.49\textwidth]{%
	\tikzsetnextfilename{./tikz/images/mixed_reg_ex1f}%
%
%

\begin{tikzpicture}[scale=1.7, transform shape]
	 \begin{scope}
		\coordinate (v0) at (0.0000000000000000,0.0000000000000000) {};
		\coordinate (v1) at (4.0000000000000000,0.0000000000000000) {};
		\coordinate (v2) at (4.0000000000000000,4.0000000000000000) {};
		\coordinate (v3) at (0.0000000000000000,4.0000000000000000) {};
		\coordinate (v4) at (1.0000000000000000,0.0000000000000000) {};
		\coordinate (v5) at (1.0000000000000000,4.0000000000000000) {};
		\coordinate (v6) at (3.0000000000000000,0.0000000000000000) {};
		\coordinate (v7) at (3.0000000000000000,4.0000000000000000) {};
		\coordinate (v8) at (0.0000000000000000,1.0000000000000000) {};
		\coordinate (v9) at (1.0000000000000000,1.0000000000000000) {};
		\coordinate (v10) at (3.0000000000000000,1.0000000000000000) {};
		\coordinate (v11) at (4.0000000000000000,1.0000000000000000) {};
		\coordinate (v12) at (0.0000000000000000,3.0000000000000000) {};
		\coordinate (v13) at (1.0000000000000000,3.0000000000000000) {};
		\coordinate (v14) at (3.0000000000000000,3.0000000000000000) {};
		\coordinate (v15) at (4.0000000000000000,3.0000000000000000) {};
		\coordinate (v16) at (2.0000000000000000,1.0000000000000000) {};
		\coordinate (v17) at (2.0000000000000000,3.0000000000000000) {};
		\coordinate (v18) at (2.0000000000000000,1.5000000000000000) {};
		\coordinate (v19) at (3.0000000000000000,1.5000000000000000) {};
		\coordinate (v20) at (2.0000000000000000,2.5000000000000000) {};
		\coordinate (v21) at (3.0000000000000000,2.5000000000000000) {};
		\coordinate (v22) at (2.5000000000000000,1.5000000000000000) {};
		\coordinate (v23) at (2.5000000000000000,2.5000000000000000) {};
		\coordinate (v24) at (1.0000000000000000,2.0000000000000000) {};
		\coordinate (v25) at (2.0000000000000000,2.0000000000000000) {};
		\coordinate (v26) at (2.5000000000000000,2.0000000000000000) {};

		 \draw[step=1,thin] (v4.center) -- (v9.center) -- (v8.center) -- (v0.center) -- cycle;
		 \draw[step=1,thin] (v6.center) -- (v10.center) -- (v16.center) -- (v9.center) -- (v4.center) -- cycle;
		 \draw[step=1,thin] (v1.center) -- (v11.center) -- (v10.center) -- (v6.center) -- cycle;
		 \draw[step=1,thin] (v9.center) -- (v24.center) -- (v13.center) -- (v12.center) -- (v8.center) -- cycle;
		 \draw[step=1,thin] (v16.center) -- (v18.center) -- (v25.center) -- (v24.center) -- (v9.center) -- cycle;
		 \draw[step=1,thin] (v11.center) -- (v15.center) -- (v14.center) -- (v21.center) -- (v19.center) -- (v10.center) -- cycle;
		 \draw[step=1,thin] (v13.center) -- (v5.center) -- (v3.center) -- (v12.center) -- cycle;
		 \draw[step=1,thin] (v17.center) -- (v14.center) -- (v7.center) -- (v5.center) -- (v13.center) -- cycle;
		 \draw[step=1,thin] (v15.center) -- (v2.center) -- (v7.center) -- (v14.center) -- cycle;
		 \draw[step=1,thin] (v10.center) -- (v19.center) -- (v22.center) -- (v18.center) -- (v16.center) -- cycle;
		 \draw[step=1,thin] (v22.center) -- (v26.center) -- (v25.center) -- (v18.center) -- cycle;
		 \draw[step=1,thin] (v23.center) -- (v21.center) -- (v14.center) -- (v17.center) -- (v20.center) -- cycle;
		 \draw[step=1,thin] (v19.center) -- (v21.center) -- (v23.center) -- (v26.center) -- (v22.center) -- cycle;
		 \draw[step=1,thin] (v25.center) -- (v20.center) -- (v17.center) -- (v13.center) -- (v24.center) -- cycle;
		 \draw[step=1,thin] (v26.center) -- (v23.center) -- (v20.center) -- (v25.center) -- cycle;

		\node[scale=0.3,draw,black,thin,fill=white,circle,minimum size=2pt,inner sep=0] at (v0) {$\vertex_{0}$};
		\node[scale=0.3,draw,black,thin,fill=white,circle,minimum size=2pt,inner sep=0] at (v1) {$\vertex_{1}$};
		\node[scale=0.3,draw,black,thin,fill=white,circle,minimum size=2pt,inner sep=0] at (v2) {$\vertex_{2}$};
		\node[scale=0.3,draw,black,thin,fill=white,circle,minimum size=2pt,inner sep=0] at (v3) {$\vertex_{3}$};
		\node[scale=0.3,draw,black,thin,fill=white,circle,minimum size=2pt,inner sep=0] at (v4) {$\vertex_{4}$};
		\node[scale=0.3,draw,black,thin,fill=white,circle,minimum size=2pt,inner sep=0] at (v5) {$\vertex_{5}$};
		\node[scale=0.3,draw,black,thin,fill=white,circle,minimum size=2pt,inner sep=0] at (v6) {$\vertex_{6}$};
		\node[scale=0.3,draw,black,thin,fill=white,circle,minimum size=2pt,inner sep=0] at (v7) {$\vertex_{7}$};
		\node[scale=0.3,draw,black,thin,fill=white,circle,minimum size=2pt,inner sep=0] at (v8) {$\vertex_{8}$};
		\node[scale=0.3,draw,black,thin,fill=white,circle,minimum size=2pt,inner sep=0] at (v9) {$\vertex_{9}$};
		\node[scale=0.3,draw,black,thin,fill=white,circle,minimum size=2pt,inner sep=0] at (v10) {$\vertex_{10}$};
		\node[scale=0.3,draw,black,thin,fill=white,circle,minimum size=2pt,inner sep=0] at (v11) {$\vertex_{11}$};
		\node[scale=0.3,draw,black,thin,fill=white,circle,minimum size=2pt,inner sep=0] at (v12) {$\vertex_{12}$};
		\node[scale=0.3,draw,black,thin,fill=white,circle,minimum size=2pt,inner sep=0] at (v13) {$\vertex_{13}$};
		\node[scale=0.3,draw,black,thin,fill=white,circle,minimum size=2pt,inner sep=0] at (v14) {$\vertex_{14}$};
		\node[scale=0.3,draw,black,thin,fill=white,circle,minimum size=2pt,inner sep=0] at (v15) {$\vertex_{15}$};
		\node[scale=0.3,draw,black,thin,fill=white,circle,minimum size=2pt,inner sep=0] at (v16) {$\vertex_{16}$};
		\node[scale=0.3,draw,black,thin,fill=white,circle,minimum size=2pt,inner sep=0] at (v17) {$\vertex_{17}$};
		\node[scale=0.3,draw,black,thin,fill=white,circle,minimum size=2pt,inner sep=0] at (v18) {$\vertex_{18}$};
		\node[scale=0.3,draw,black,thin,fill=white,circle,minimum size=2pt,inner sep=0] at (v19) {$\vertex_{19}$};
		\node[scale=0.3,draw,black,thin,fill=white,circle,minimum size=2pt,inner sep=0] at (v20) {$\vertex_{20}$};
		\node[scale=0.3,draw,black,thin,fill=white,circle,minimum size=2pt,inner sep=0] at (v21) {$\vertex_{21}$};
		\node[scale=0.3,draw,black,thin,fill=white,circle,minimum size=2pt,inner sep=0] at (v22) {$\vertex_{22}$};
		\node[scale=0.3,draw,black,thin,fill=white,circle,minimum size=2pt,inner sep=0] at (v23) {$\vertex_{23}$};
		\node[scale=0.3,draw,black,thin,fill=white,circle,minimum size=2pt,inner sep=0] at (v24) {$\vertex_{24}$};
		\node[scale=0.3,draw,black,thin,fill=white,circle,minimum size=2pt,inner sep=0] at (v25) {$\vertex_{25}$};
		\node[scale=0.3,draw,black,thin,fill=white,circle,minimum size=2pt,inner sep=0] at (v26) {$\vertex_{26}$};

		\node[scale=0.3,below,inner sep=0] at ($(v0)!0.5!(v4)$) {$\edge_{0}(-1)$};
		\node[scale=0.3,left,inner sep=0] at ($(v1)!0.5!(v11)$) {$\edge_{1}(-1)$};
		\node[scale=0.3,below,inner sep=0] at ($(v3)!0.5!(v5)$) {$\edge_{2}(-1)$};
		\node[scale=0.3,left,inner sep=0] at ($(v0)!0.5!(v8)$) {$\edge_{3}(-1)$};
		\node[scale=0.3,below,inner sep=0] at ($(v4)!0.5!(v6)$) {$\edge_{4}(-1)$};
		\node[scale=0.3,below,inner sep=0] at ($(v5)!0.5!(v7)$) {$\edge_{5}(-1)$};
		\node[scale=0.3,left,inner sep=0] at ($(v4)!0.5!(v9)$) {$\edge_{6}(2)$};
		\node[scale=0.3,below,inner sep=0] at ($(v6)!0.5!(v1)$) {$\edge_{7}(-1)$};
		\node[scale=0.3,below,inner sep=0] at ($(v7)!0.5!(v2)$) {$\edge_{8}(-1)$};
		\node[scale=0.3,left,inner sep=0] at ($(v6)!0.5!(v10)$) {$\edge_{9}(1)$};
		\node[scale=0.3,left,inner sep=0] at ($(v8)!0.5!(v12)$) {$\edge_{10}(-1)$};
		\node[scale=0.3,left,inner sep=0] at ($(v9)!0.5!(v24)$) {$\edge_{11}(2)$};
		\node[scale=0.3,below,inner sep=0] at ($(v8)!0.5!(v9)$) {$\edge_{12}(0)$};
		\node[scale=0.3,left,inner sep=0] at ($(v10)!0.5!(v19)$) {$\edge_{13}(1)$};
		\node[scale=0.3,below,inner sep=0] at ($(v9)!0.5!(v16)$) {$\edge_{14}(0)$};
		\node[scale=0.3,left,inner sep=0] at ($(v11)!0.5!(v15)$) {$\edge_{15}(-1)$};
		\node[scale=0.3,below,inner sep=0] at ($(v10)!0.5!(v11)$) {$\edge_{16}(0)$};
		\node[scale=0.3,left,inner sep=0] at ($(v12)!0.5!(v3)$) {$\edge_{17}(-1)$};
		\node[scale=0.3,left,inner sep=0] at ($(v13)!0.5!(v5)$) {$\edge_{18}(2)$};
		\node[scale=0.3,below,inner sep=0] at ($(v12)!0.5!(v13)$) {$\edge_{19}(2)$};
		\node[scale=0.3,left,inner sep=0] at ($(v14)!0.5!(v7)$) {$\edge_{20}(1)$};
		\node[scale=0.3,below,inner sep=0] at ($(v13)!0.5!(v17)$) {$\edge_{21}(2)$};
		\node[scale=0.3,left,inner sep=0] at ($(v15)!0.5!(v2)$) {$\edge_{22}(-1)$};
		\node[scale=0.3,below,inner sep=0] at ($(v14)!0.5!(v15)$) {$\edge_{23}(2)$};
		\node[scale=0.3,below,inner sep=0] at ($(v16)!0.5!(v10)$) {$\edge_{24}(0)$};
		\node[scale=0.3,below,inner sep=0] at ($(v17)!0.5!(v14)$) {$\edge_{25}(2)$};
		\node[scale=0.3,left,inner sep=0] at ($(v16)!0.5!(v18)$) {$\edge_{26}(1)$};
		\node[scale=0.3,left,inner sep=0] at ($(v18)!0.5!(v25)$) {$\edge_{27}(1)$};
		\node[scale=0.3,left,inner sep=0] at ($(v19)!0.5!(v21)$) {$\edge_{28}(1)$};
		\node[scale=0.3,below,inner sep=0] at ($(v18)!0.5!(v22)$) {$\edge_{29}(0)$};
		\node[scale=0.3,left,inner sep=0] at ($(v20)!0.5!(v17)$) {$\color{blue}\edge_{30}(0)$};
		\node[scale=0.3,left,inner sep=0] at ($(v21)!0.5!(v14)$) {$\edge_{31}(1)$};
		\node[scale=0.3,below,inner sep=0] at ($(v20)!0.5!(v23)$) {$\edge_{32}(2)$};
		\node[scale=0.3,below,inner sep=0] at ($(v22)!0.5!(v19)$) {$\edge_{33}(0)$};
		\node[scale=0.3,below,inner sep=0] at ($(v23)!0.5!(v21)$) {$\edge_{34}(2)$};
		\node[scale=0.3,left,inner sep=0] at ($(v22)!0.5!(v26)$) {$\edge_{35}(2)$};
		\node[scale=0.3,left,inner sep=0] at ($(v24)!0.5!(v13)$) {$\edge_{36}(2)$};
		\node[scale=0.3,left,inner sep=0] at ($(v25)!0.5!(v20)$) {$\color{blue}\edge_{37}(0)$};
		\node[scale=0.3,below,inner sep=0] at ($(v24)!0.5!(v25)$) {$\edge_{38}(2)$};
		\node[scale=0.3,left,inner sep=0] at ($(v26)!0.5!(v23)$) {$\edge_{39}(2)$};
		\node[scale=0.3,below,inner sep=0] at ($(v25)!0.5!(v26)$) {$\edge_{40}(2)$};

	 \end{scope}
\end{tikzpicture}

%
}
	\caption{(Continued from previous page.) The above figures correspond to the bi-cubic spline space considered in Example \ref{ex:ex2}.
		The smoothness required across each edge has been annotated in parenthesis next to the edge labels.
		The smoothness distributions in Figures (b)--(f) differ from the one in Figure (a) only on the edges labelled in blue.
	}
\end{figure*}

This section presents examples of settings where Theorem \ref{thm:dimension} applies, and also where it does not.
In particular, we show that $H_0(\idealComplex^{\bsmoothr})$ can be non-trivial when the conditions of Theorem \ref{thm:dimension} are not met.

\begin{example}[PHT-splines of mixed smoothness]\label{ex:ex1}
	Consider the PHT-spline space $\splSpace^{\bsmooth}_{33}$ shown in Figure \ref{fig:ex1}(a).
	From Corollary \ref{cor:pht}, we can reduce the smoothness across any arbitrary edge and the dimension will still be given by the Euler characteristic of $\quotientComplex^{\bsmooth}$.
	One such modification is shown in Figure \ref{fig:ex1}(b) where the smoothness across edges $\edge_{27}, \edge_{34}, \edge_{35}$ and $\edge_{36}$ have been reduced.
	The dimensions of the spaces can be easily computed to be the following,
	\begin{equation*}
		\dimwp{\splSpace^{\bsmooth}} = 64\;,\qquad
		\dimwp{\splSpace^{\bsmoothr}} = 66\;.
	\end{equation*}
\end{example}

\begin{example}[Splines of mixed smoothness; hierarchical T-mesh]\label{ex:ex2}
	Consider the space of bi-cubic splines $\splSpace^{\bsmooth}_{33}$ for the smoothness distribution shown in Figure \ref{fig:ex2}(a).
	In Figures (b)--(e), we successively reduce the smoothness across the edges labelled in blue while ensuring that the conditions of Theorem \ref{thm:dimension} are met.
	As a result, at each step of smoothness reduction, we have $H_0(\idealComplex^{\bsmoothr_i}) = 0$, $i = 0, \dots, 3$.
	The dimensions of the corresponding spline spaces can then be easily computed using the Euler characteristics of complexes $\quotientComplex^{\bsmoothr_i}$,
	\begin{equation*}
	\begin{split}
		&\dimwp{\splSpace^{\bsmooth}} = 56\;,~~
		\dimwp{\splSpace^{\bsmoothr_0}} = 57\;,~~
		\dimwp{\splSpace^{\bsmoothr_1}} = 58\;,\\
		&\dimwp{\splSpace^{\bsmoothr_2}} = 58\;,~~
		\dimwp{\splSpace^{\bsmoothr_3}} = 59\;.
	\end{split}
	\end{equation*}
	On the other hand, reducing the smoothness from Figure (a) to Figure (f) does not satisfy the conditions of Theorem \ref{thm:dimension}.
	Indeed, for $A = \{ \edge_{30}, \edge_{37}, \vertex_{17}, \vertex_{20}, \vertex_{25} \}$, it can be verified that $\omega^{\bsmoothr_4}(A) = 3 < \degreev + 1 = 4$.
	In this case, it can also be computed (using Macaulay2 \cite{M2}, for instance) that $\dimwp{H_0(\idealComplex^{\bsmoothr_4})} = 1$.
\end{example}

\begin{figure}[t]
	\subcaptionbox{$\splSpace^{\bsmooth}$}[0.49\textwidth]{%
	\tikzsetnextfilename{./tikz/images/mixed_reg_ex3a_tcycle}%
%
%

\begin{tikzpicture}[scale=1.7, transform shape]
	 \begin{scope}
		\coordinate (v0) at (0.0000000000000000,0.0000000000000000) {};
		\coordinate (v1) at (4.0000000000000000,0.0000000000000000) {};
		\coordinate (v2) at (4.0000000000000000,4.0000000000000000) {};
		\coordinate (v3) at (0.0000000000000000,4.0000000000000000) {};
		\coordinate (v4) at (0.6666666666666666,0.0000000000000000) {};
		\coordinate (v5) at (0.6666666666666666,4.0000000000000000) {};
		\coordinate (v6) at (3.3333333333333335,0.0000000000000000) {};
		\coordinate (v7) at (3.3333333333333335,4.0000000000000000) {};
		\coordinate (v8) at (0.0000000000000000,0.6666666666666666) {};
		\coordinate (v9) at (0.6666666666666666,0.6666666666666666) {};
		\coordinate (v10) at (3.3333333333333335,0.6666666666666666) {};
		\coordinate (v11) at (4.0000000000000000,0.6666666666666666) {};
		\coordinate (v12) at (0.0000000000000000,3.3333333333333335) {};
		\coordinate (v13) at (0.6666666666666666,3.3333333333333335) {};
		\coordinate (v14) at (3.3333333333333335,3.3333333333333335) {};
		\coordinate (v15) at (4.0000000000000000,3.3333333333333335) {};
		\coordinate (v16) at (0.6666666666666666,1.3333333333333333) {};
		\coordinate (v17) at (3.3333333333333335,1.3333333333333333) {};
		\coordinate (v18) at (1.3333333333333333,1.3333333333333333) {};
		\coordinate (v19) at (1.3333333333333333,3.3333333333333335) {};
		\coordinate (v20) at (1.3333333333333333,2.6666666666666665) {};
		\coordinate (v21) at (3.3333333333333335,2.6666666666666665) {};
		\coordinate (v22) at (2.6666666666666665,0.6666666666666666) {};
		\coordinate (v23) at (2.6666666666666665,1.3333333333333333) {};
		\coordinate (v24) at (2.6666666666666665,2.6666666666666665) {};

		 \draw[step=1,thin] (v4.center) -- (v9.center) -- (v8.center) -- (v0.center) -- cycle;
		 \draw[step=1,thin] (v6.center) -- (v10.center) -- (v22.center) -- (v9.center) -- (v4.center) -- cycle;
		 \draw[step=1,thin] (v1.center) -- (v11.center) -- (v10.center) -- (v6.center) -- cycle;
		 \draw[step=1,thin] (v9.center) -- (v16.center) -- (v13.center) -- (v12.center) -- (v8.center) -- cycle;
		 \draw[step=1,thin] (v22.center) -- (v23.center) -- (v18.center) -- (v16.center) -- (v9.center) -- cycle;
		 \draw[step=1,thin] (v11.center) -- (v15.center) -- (v14.center) -- (v21.center) -- (v17.center) -- (v10.center) -- cycle;
		 \draw[step=1,thin] (v13.center) -- (v5.center) -- (v3.center) -- (v12.center) -- cycle;
		 \draw[step=1,thin] (v19.center) -- (v14.center) -- (v7.center) -- (v5.center) -- (v13.center) -- cycle;
		 \draw[step=1,thin] (v15.center) -- (v2.center) -- (v7.center) -- (v14.center) -- cycle;
		 \draw[step=1,thin] (v18.center) -- (v20.center) -- (v19.center) -- (v13.center) -- (v16.center) -- cycle;
		 \draw[step=1,thin] (v23.center) -- (v24.center) -- (v20.center) -- (v18.center) -- cycle;
		 \draw[step=1,thin] (v24.center) -- (v21.center) -- (v14.center) -- (v19.center) -- (v20.center) -- cycle;

		\node[scale=0.3,draw,black,thin,fill=white,circle,minimum size=2pt,inner sep=0] at (v0) {$\vertex_{0}$};
		\node[scale=0.3,draw,black,thin,fill=white,circle,minimum size=2pt,inner sep=0] at (v1) {$\vertex_{1}$};
		\node[scale=0.3,draw,black,thin,fill=white,circle,minimum size=2pt,inner sep=0] at (v2) {$\vertex_{2}$};
		\node[scale=0.3,draw,black,thin,fill=white,circle,minimum size=2pt,inner sep=0] at (v3) {$\vertex_{3}$};
		\node[scale=0.3,draw,black,thin,fill=white,circle,minimum size=2pt,inner sep=0] at (v4) {$\vertex_{4}$};
		\node[scale=0.3,draw,black,thin,fill=white,circle,minimum size=2pt,inner sep=0] at (v5) {$\vertex_{5}$};
		\node[scale=0.3,draw,black,thin,fill=white,circle,minimum size=2pt,inner sep=0] at (v6) {$\vertex_{6}$};
		\node[scale=0.3,draw,black,thin,fill=white,circle,minimum size=2pt,inner sep=0] at (v7) {$\vertex_{7}$};
		\node[scale=0.3,draw,black,thin,fill=white,circle,minimum size=2pt,inner sep=0] at (v8) {$\vertex_{8}$};
		\node[scale=0.3,draw,black,thin,fill=white,circle,minimum size=2pt,inner sep=0] at (v9) {$\vertex_{9}$};
		\node[scale=0.3,draw,black,thin,fill=white,circle,minimum size=2pt,inner sep=0] at (v10) {$\vertex_{10}$};
		\node[scale=0.3,draw,black,thin,fill=white,circle,minimum size=2pt,inner sep=0] at (v11) {$\vertex_{11}$};
		\node[scale=0.3,draw,black,thin,fill=white,circle,minimum size=2pt,inner sep=0] at (v12) {$\vertex_{12}$};
		\node[scale=0.3,draw,black,thin,fill=white,circle,minimum size=2pt,inner sep=0] at (v13) {$\vertex_{13}$};
		\node[scale=0.3,draw,black,thin,fill=white,circle,minimum size=2pt,inner sep=0] at (v14) {$\vertex_{14}$};
		\node[scale=0.3,draw,black,thin,fill=white,circle,minimum size=2pt,inner sep=0] at (v15) {$\vertex_{15}$};
		\node[scale=0.3,draw,black,thin,fill=white,circle,minimum size=2pt,inner sep=0] at (v16) {$\vertex_{16}$};
		\node[scale=0.3,draw,black,thin,fill=white,circle,minimum size=2pt,inner sep=0] at (v18) {$\vertex_{17}$};
		\node[scale=0.3,draw,black,thin,fill=white,circle,minimum size=2pt,inner sep=0] at (v19) {$\vertex_{18}$};
		\node[scale=0.3,draw,black,thin,fill=white,circle,minimum size=2pt,inner sep=0] at (v20) {$\vertex_{19}$};
		\node[scale=0.3,draw,black,thin,fill=white,circle,minimum size=2pt,inner sep=0] at (v21) {$\vertex_{20}$};
		\node[scale=0.3,draw,black,thin,fill=white,circle,minimum size=2pt,inner sep=0] at (v22) {$\vertex_{21}$};
		\node[scale=0.3,draw,black,thin,fill=white,circle,minimum size=2pt,inner sep=0] at (v23) {$\vertex_{22}$};
		\node[scale=0.3,draw,black,thin,fill=white,circle,minimum size=2pt,inner sep=0] at (v24) {$\vertex_{23}$};

		\node[scale=0.3,below,inner sep=0] at ($(v0)!0.5!(v4)$) {$\edge_{0}(-1)$};
		\node[scale=0.3,left,inner sep=0] at ($(v1)!0.5!(v11)$) {$\edge_{1}(-1)$};
		\node[scale=0.3,below,inner sep=0] at ($(v3)!0.5!(v5)$) {$\edge_{2}(-1)$};
		\node[scale=0.3,left,inner sep=0] at ($(v0)!0.5!(v8)$) {$\edge_{3}(-1)$};
		\node[scale=0.3,below,inner sep=0] at ($(v4)!0.5!(v6)$) {$\edge_{4}(-1)$};
		\node[scale=0.3,below,inner sep=0] at ($(v5)!0.5!(v7)$) {$\edge_{5}(-1)$};
		\node[scale=0.3,left,inner sep=0] at ($(v4)!0.5!(v9)$) {$\edge_{6}(1)$};
		\node[scale=0.3,below,inner sep=0] at ($(v6)!0.5!(v1)$) {$\edge_{7}(-1)$};
		\node[scale=0.3,below,inner sep=0] at ($(v7)!0.5!(v2)$) {$\edge_{8}(-1)$};
		\node[scale=0.3,left,inner sep=0] at ($(v6)!0.5!(v10)$) {$\edge_{9}(1)$};
		\node[scale=0.3,left,inner sep=0] at ($(v8)!0.5!(v12)$) {$\edge_{10}(-1)$};
		\node[scale=0.3,left,inner sep=0] at ($(v9)!0.5!(v16)$) {$\edge_{11}(1)$};
		\node[scale=0.3,below,inner sep=0] at ($(v8)!0.5!(v9)$) {$\edge_{12}(1)$};
		\node[scale=0.3,left,inner sep=0] at ($(v10)!0.5!(v21)$) {$\edge_{13}(1)$};
		\node[scale=0.3,below,inner sep=0] at ($(v9)!0.5!(v22)$) {$\edge_{14}(1)$};
		\node[scale=0.3,left,inner sep=0] at ($(v11)!0.5!(v15)$) {$\edge_{15}(-1)$};
		\node[scale=0.3,below,inner sep=0] at ($(v10)!0.5!(v11)$) {$\edge_{16}(1)$};
		\node[scale=0.3,left,inner sep=0] at ($(v12)!0.5!(v3)$) {$\edge_{17}(-1)$};
		\node[scale=0.3,left,inner sep=0] at ($(v13)!0.5!(v5)$) {$\edge_{18}(1)$};
		\node[scale=0.3,below,inner sep=0] at ($(v12)!0.5!(v13)$) {$\edge_{19}(1)$};
		\node[scale=0.3,left,inner sep=0] at ($(v14)!0.5!(v7)$) {$\edge_{20}(1)$};
		\node[scale=0.3,below,inner sep=0] at ($(v13)!0.5!(v19)$) {$\edge_{21}(1)$};
		\node[scale=0.3,left,inner sep=0] at ($(v15)!0.5!(v2)$) {$\edge_{22}(-1)$};
		\node[scale=0.3,below,inner sep=0] at ($(v14)!0.5!(v15)$) {$\edge_{23}(1)$};
		\node[scale=0.3,left,inner sep=0] at ($(v16)!0.5!(v13)$) {$\edge_{24}(1)$};
		\node[scale=0.3,below,inner sep=0] at ($(v16)!0.5!(v18)$) {$\edge_{25}(1)$};
		\node[scale=0.3,below,inner sep=0] at ($(v18)!0.5!(v23)$) {$\edge_{26}(1)$};
		\node[scale=0.3,below,inner sep=0] at ($(v19)!0.5!(v14)$) {$\edge_{27}(1)$};
		\node[scale=0.3,left,inner sep=0] at ($(v18)!0.5!(v20)$) {$\edge_{28}(1)$};
		\node[scale=0.3,left,inner sep=0] at ($(v20)!0.5!(v19)$) {$\edge_{29}(1)$};
		\node[scale=0.3,left,inner sep=0] at ($(v21)!0.5!(v14)$) {$\edge_{30}(1)$};
		\node[scale=0.3,below,inner sep=0] at ($(v20)!0.5!(v24)$) {$\edge_{31}(1)$};
		\node[scale=0.3,below,inner sep=0] at ($(v22)!0.5!(v10)$) {$\edge_{32}(1)$};
		\node[scale=0.3,left,inner sep=0] at ($(v22)!0.5!(v23)$) {$\edge_{33}(1)$};
		\node[scale=0.3,below,inner sep=0] at ($(v24)!0.5!(v21)$) {$\edge_{34}(1)$};
		\node[scale=0.3,left,inner sep=0] at ($(v23)!0.5!(v24)$) {$\edge_{35}(1)$};

	 \end{scope}
\end{tikzpicture}

%
}
	\subcaptionbox{$\splSpace^{\bsmoothr}$}[0.49\textwidth]{%
	\tikzsetnextfilename{./tikz/images/mixed_reg_ex3c_tcycle}%
%
%
	
	\begin{tikzpicture}[scale=1.7, transform shape]
	\begin{scope}
	\coordinate (v0) at (0.0000000000000000,0.0000000000000000) {};
	\coordinate (v1) at (4.0000000000000000,0.0000000000000000) {};
	\coordinate (v2) at (4.0000000000000000,4.0000000000000000) {};
	\coordinate (v3) at (0.0000000000000000,4.0000000000000000) {};
	\coordinate (v4) at (0.6666666666666666,0.0000000000000000) {};
	\coordinate (v5) at (0.6666666666666666,4.0000000000000000) {};
	\coordinate (v6) at (3.3333333333333335,0.0000000000000000) {};
	\coordinate (v7) at (3.3333333333333335,4.0000000000000000) {};
	\coordinate (v8) at (0.0000000000000000,0.6666666666666666) {};
	\coordinate (v9) at (0.6666666666666666,0.6666666666666666) {};
	\coordinate (v10) at (3.3333333333333335,0.6666666666666666) {};
	\coordinate (v11) at (4.0000000000000000,0.6666666666666666) {};
	\coordinate (v12) at (0.0000000000000000,3.3333333333333335) {};
	\coordinate (v13) at (0.6666666666666666,3.3333333333333335) {};
	\coordinate (v14) at (3.3333333333333335,3.3333333333333335) {};
	\coordinate (v15) at (4.0000000000000000,3.3333333333333335) {};
	\coordinate (v16) at (0.6666666666666666,1.3333333333333333) {};
	\coordinate (v17) at (3.3333333333333335,1.3333333333333333) {};
	\coordinate (v18) at (1.3333333333333333,1.3333333333333333) {};
	\coordinate (v19) at (1.3333333333333333,3.3333333333333335) {};
	\coordinate (v20) at (1.3333333333333333,2.6666666666666665) {};
	\coordinate (v21) at (3.3333333333333335,2.6666666666666665) {};
	\coordinate (v22) at (2.6666666666666665,0.6666666666666666) {};
	\coordinate (v23) at (2.6666666666666665,1.3333333333333333) {};
	\coordinate (v24) at (2.6666666666666665,2.6666666666666665) {};
	
	\draw[step=1,thin] (v4.center) -- (v9.center) -- (v8.center) -- (v0.center) -- cycle;
	\draw[step=1,thin] (v6.center) -- (v10.center) -- (v22.center) -- (v9.center) -- (v4.center) -- cycle;
	\draw[step=1,thin] (v1.center) -- (v11.center) -- (v10.center) -- (v6.center) -- cycle;
	\draw[step=1,thin] (v9.center) -- (v16.center) -- (v13.center) -- (v12.center) -- (v8.center) -- cycle;
	\draw[step=1,thin] (v22.center) -- (v23.center) -- (v18.center) -- (v16.center) -- (v9.center) -- cycle;
	\draw[step=1,thin] (v11.center) -- (v15.center) -- (v14.center) -- (v21.center) -- (v17.center) -- (v10.center) -- cycle;
	\draw[step=1,thin] (v13.center) -- (v5.center) -- (v3.center) -- (v12.center) -- cycle;
	\draw[step=1,thin] (v19.center) -- (v14.center) -- (v7.center) -- (v5.center) -- (v13.center) -- cycle;
	\draw[step=1,thin] (v15.center) -- (v2.center) -- (v7.center) -- (v14.center) -- cycle;
	\draw[step=1,thin] (v18.center) -- (v20.center) -- (v19.center) -- (v13.center) -- (v16.center) -- cycle;
	\draw[step=1,thin] (v23.center) -- (v24.center) -- (v20.center) -- (v18.center) -- cycle;
	\draw[step=1,thin] (v24.center) -- (v21.center) -- (v14.center) -- (v19.center) -- (v20.center) -- cycle;
	
	\node[scale=0.3,draw,black,thin,fill=white,circle,minimum size=2pt,inner sep=0] at (v0) {$\vertex_{0}$};
	\node[scale=0.3,draw,black,thin,fill=white,circle,minimum size=2pt,inner sep=0] at (v1) {$\vertex_{1}$};
	\node[scale=0.3,draw,black,thin,fill=white,circle,minimum size=2pt,inner sep=0] at (v2) {$\vertex_{2}$};
	\node[scale=0.3,draw,black,thin,fill=white,circle,minimum size=2pt,inner sep=0] at (v3) {$\vertex_{3}$};
	\node[scale=0.3,draw,black,thin,fill=white,circle,minimum size=2pt,inner sep=0] at (v4) {$\vertex_{4}$};
	\node[scale=0.3,draw,black,thin,fill=white,circle,minimum size=2pt,inner sep=0] at (v5) {$\vertex_{5}$};
	\node[scale=0.3,draw,black,thin,fill=white,circle,minimum size=2pt,inner sep=0] at (v6) {$\vertex_{6}$};
	\node[scale=0.3,draw,black,thin,fill=white,circle,minimum size=2pt,inner sep=0] at (v7) {$\vertex_{7}$};
	\node[scale=0.3,draw,black,thin,fill=white,circle,minimum size=2pt,inner sep=0] at (v8) {$\vertex_{8}$};
	\node[scale=0.3,draw,black,thin,fill=white,circle,minimum size=2pt,inner sep=0] at (v9) {$\vertex_{9}$};
	\node[scale=0.3,draw,black,thin,fill=white,circle,minimum size=2pt,inner sep=0] at (v10) {$\vertex_{10}$};
	\node[scale=0.3,draw,black,thin,fill=white,circle,minimum size=2pt,inner sep=0] at (v11) {$\vertex_{11}$};
	\node[scale=0.3,draw,black,thin,fill=white,circle,minimum size=2pt,inner sep=0] at (v12) {$\vertex_{12}$};
	\node[scale=0.3,draw,black,thin,fill=white,circle,minimum size=2pt,inner sep=0] at (v13) {$\vertex_{13}$};
	\node[scale=0.3,draw,black,thin,fill=white,circle,minimum size=2pt,inner sep=0] at (v14) {$\vertex_{14}$};
	\node[scale=0.3,draw,black,thin,fill=white,circle,minimum size=2pt,inner sep=0] at (v15) {$\vertex_{15}$};
	\node[scale=0.3,draw,black,thin,fill=white,circle,minimum size=2pt,inner sep=0] at (v16) {$\vertex_{16}$};
	\node[scale=0.3,draw,black,thin,fill=white,circle,minimum size=2pt,inner sep=0] at (v18) {$\vertex_{17}$};
	\node[scale=0.3,draw,black,thin,fill=white,circle,minimum size=2pt,inner sep=0] at (v19) {$\vertex_{18}$};
	\node[scale=0.3,draw,black,thin,fill=white,circle,minimum size=2pt,inner sep=0] at (v20) {$\vertex_{19}$};
	\node[scale=0.3,draw,black,thin,fill=white,circle,minimum size=2pt,inner sep=0] at (v21) {$\vertex_{20}$};
	\node[scale=0.3,draw,black,thin,fill=white,circle,minimum size=2pt,inner sep=0] at (v22) {$\vertex_{21}$};
	\node[scale=0.3,draw,black,thin,fill=white,circle,minimum size=2pt,inner sep=0] at (v23) {$\vertex_{22}$};
	\node[scale=0.3,draw,black,thin,fill=white,circle,minimum size=2pt,inner sep=0] at (v24) {$\vertex_{23}$};
	
	\node[scale=0.3,below,inner sep=0] at ($(v0)!0.5!(v4)$) {$\edge_{0}(-1)$};
	\node[scale=0.3,left,inner sep=0] at ($(v1)!0.5!(v11)$) {$\edge_{1}(-1)$};
	\node[scale=0.3,below,inner sep=0] at ($(v3)!0.5!(v5)$) {$\edge_{2}(-1)$};
	\node[scale=0.3,left,inner sep=0] at ($(v0)!0.5!(v8)$) {$\edge_{3}(-1)$};
	\node[scale=0.3,below,inner sep=0] at ($(v4)!0.5!(v6)$) {$\edge_{4}(-1)$};
	\node[scale=0.3,below,inner sep=0] at ($(v5)!0.5!(v7)$) {$\edge_{5}(-1)$};
	\node[scale=0.3,left,inner sep=0] at ($(v4)!0.5!(v9)$) {$\edge_{6}(1)$};
	\node[scale=0.3,below,inner sep=0] at ($(v6)!0.5!(v1)$) {$\edge_{7}(-1)$};
	\node[scale=0.3,below,inner sep=0] at ($(v7)!0.5!(v2)$) {$\edge_{8}(-1)$};
	\node[scale=0.3,left,inner sep=0] at ($(v6)!0.5!(v10)$) {$\edge_{9}(1)$};
	\node[scale=0.3,left,inner sep=0] at ($(v8)!0.5!(v12)$) {$\edge_{10}(-1)$};
	\node[scale=0.3,left,inner sep=0] at ($(v9)!0.5!(v16)$) {$\edge_{11}(1)$};
	\node[scale=0.3,below,inner sep=0] at ($(v8)!0.5!(v9)$) {$\edge_{12}(1)$};
	\node[scale=0.3,left,inner sep=0] at ($(v10)!0.5!(v21)$) {$\edge_{13}(1)$};
	\node[scale=0.3,below,inner sep=0] at ($(v9)!0.5!(v22)$) {$\edge_{14}(1)$};
	\node[scale=0.3,left,inner sep=0] at ($(v11)!0.5!(v15)$) {$\edge_{15}(-1)$};
	\node[scale=0.3,below,inner sep=0] at ($(v10)!0.5!(v11)$) {$\edge_{16}(1)$};
	\node[scale=0.3,left,inner sep=0] at ($(v12)!0.5!(v3)$) {$\edge_{17}(-1)$};
	\node[scale=0.3,left,inner sep=0] at ($(v13)!0.5!(v5)$) {$\edge_{18}(1)$};
	\node[scale=0.3,below,inner sep=0] at ($(v12)!0.5!(v13)$) {$\edge_{19}(1)$};
	\node[scale=0.3,left,inner sep=0] at ($(v14)!0.5!(v7)$) {$\edge_{20}(1)$};
	\node[scale=0.3,below,inner sep=0] at ($(v13)!0.5!(v19)$) {$\color{blue}\edge_{21}(0)$};
	\node[scale=0.3,left,inner sep=0] at ($(v15)!0.5!(v2)$) {$\edge_{22}(-1)$};
	\node[scale=0.3,below,inner sep=0] at ($(v14)!0.5!(v15)$) {$\edge_{23}(1)$};
	\node[scale=0.3,left,inner sep=0] at ($(v16)!0.5!(v13)$) {$\edge_{24}(1)$};
	\node[scale=0.3,below,inner sep=0] at ($(v16)!0.5!(v18)$) {$\color{blue}\edge_{25}(-1)$};
	\node[scale=0.3,below,inner sep=0] at ($(v18)!0.5!(v23)$) {$\color{blue}\edge_{26}(-1)$};
	\node[scale=0.3,below,inner sep=0] at ($(v19)!0.5!(v14)$) {$\color{blue}\edge_{27}(0)$};
	\node[scale=0.3,left,inner sep=0] at ($(v18)!0.5!(v20)$) {$\edge_{28}(1)$};
	\node[scale=0.3,left,inner sep=0] at ($(v20)!0.5!(v19)$) {$\color{blue}\edge_{29}(0)$};
	\node[scale=0.3,left,inner sep=0] at ($(v21)!0.5!(v14)$) {$\edge_{30}(1)$};
	\node[scale=0.3,below,inner sep=0] at ($(v20)!0.5!(v24)$) {$\edge_{31}(1)$};
	\node[scale=0.3,below,inner sep=0] at ($(v22)!0.5!(v10)$) {$\edge_{32}(1)$};
	\node[scale=0.3,left,inner sep=0] at ($(v22)!0.5!(v23)$) {$\edge_{33}(1)$};
	\node[scale=0.3,below,inner sep=0] at ($(v24)!0.5!(v21)$) {$\edge_{34}(1)$};
	\node[scale=0.3,left,inner sep=0] at ($(v23)!0.5!(v24)$) {$\edge_{35}(1)$};
	
	\end{scope}
	\end{tikzpicture}
	
%
}
	\caption{The above figures correspond to the setting considered in Example \ref{ex:ex3}.
		The smoothness required across each edge has been annotated in parenthesis next to the edge labels.
		Figure (a) shows the initial smoothness distribution, while Figure (b) shows the modified initial distributions; the modifications are limited to the edges labelled in blue.
	}
	\label{fig:ex3}
\end{figure}

\begin{example}[Splines of mixed smoothness; non-hierarchical T-mesh]\label{ex:ex3}
	Consider the space of bi-cubic splines $\splSpace^{\bsmooth}_{33}$ for the smoothness distribution shown in Figure \ref{fig:ex3}(a).
	Note that in this case the T-mesh cannot be constructed hierarchically.
	Nevertheless, it is possible to use results from \cite{li2019instability} to verify that $H_0(\idealComplex^{\bsmooth}) = 0$ and $\dimwp{\splSpace^{\bsmooth}} = \euler{\quotientComplex^{\bsmooth}} = 64$.
	Then, using Theorem \ref{thm:pht}, we see that we can reduce the smoothness across any subset of edges and maintain $H_0(\idealComplex^{\bsmoothr}) = 0$ for the new smoothness distribution $\bsmoothr$.
	One such case has been shown in Figure \ref{fig:ex3}(b), and the corresponding dimension of the space is given by $\dimwp{\splSpace^{\bsmoothr}} = \euler{\quotientComplex^{\bsmoothr}} = 71$.
\end{example}
	
	\section{Conclusions}\label{sec:conclusions}

Smooth polynomial splines are immensely versatile and are routinely utilized for challenging applications in, for instance, geometric modelling and computational analysis.
However, certain tasks also require working with splines of reduced smoothness, at least locally; e.g., geometric objects containing $C^0$ feature lines, solutions to physical problems that show localized discontinuities.
Local control over the smoothness can be very beneficial in such cases and can lead to great improvements in the quality of the output.
In this paper we have studied the dimension of bi-degree splines on T-meshes when different orders of smoothness are required across different mesh edges.
Reducing the problem to an essentially univariate problem, we have provided sufficient conditions that ensure that the dimension can be combinatorially computed using only local information.
The conditions are constructive in nature and have simple geometric interpretation.
A forthcoming paper will focus on the construction of a normalized B-spline-like basis for such spline spaces.
	
	\bibliographystyle{plain}
	\bibliography{bibliography}

\end{document}